\documentclass[12pt]{amsart}
\usepackage{amssymb,amscd}
\usepackage{amsbsy,amsmath,amsfonts,amsthm,extarrows,bm}
\usepackage[english]{babel}
\usepackage{graphicx}
\usepackage[all]{xy}
\usepackage{xfrac}

\usepackage[utf8]{inputenc}
\usepackage[T1]{fontenc}
\usepackage{tikz-cd} 

\usepackage{xr}
\usepackage[bookmarksopen=false,breaklinks=true,backref=page,pagebackref=true,plainpages=false, hyperindex=true,pdfstartview=FitH,colorlinks=true, pdfpagelabels=true,linkcolor=blue,citecolor=red,urlcolor=red,hypertexnames=false]{hyperref}

\newtheorem{theorem}{Theorem}[section]
\newtheorem{lemma}[theorem]{Lemma}
\newtheorem{corollary}[theorem]{Corollary}
\newtheorem{proposition}[theorem]{Proposition}
\newtheorem{remark}[theorem]{Remark}

\newtheorem{definition}[theorem]{Definition}
\newtheorem{example}[theorem]{Example}
\newtheorem{examples}[theorem]{Examples}

\newtheorem{Atheorem}{Theorem}[section]

\newtheorem{Acorollary}[Atheorem]{Corollary}

\newtheorem{Aproposition}[Atheorem]{Proposition}

\newtheorem{Adefinition}[Atheorem]{Definition}

\newenvironment{sproof}[1]
{\begin{proof}[#1]} {\end{proof}}

\newcommand{\Z}{\mathbb Z}

\newcommand{\Q}{\mathbb Q}

\newcommand{\GL}{\textnormal{GL}}
\newcommand{\SL}{\textnormal{SL}}

\newcommand{\El}{\textnormal{E}}

\newcommand{\sr}{\textnormal{sr}}
\newcommand{\glr}{\textnormal{glr}}
\newcommand{\er}{\textnormal{er}}

\newcommand{\GE}{\textnormal{GE}}

\newcommand{\W}{\textnormal{W}}
\newcommand{\E}{\textnormal{E}}

\newcommand{\Aut}{\textnormal{Aut}}

\newcommand{\V}{\operatorname{V}}

\newcommand{\Jac}{\mathcal{J}}

\newcommand{\SK}{\operatorname{SK}}

\newcommand{\Fitt}{\operatorname{Fitt}}
\newcommand{\Rp}{R_{\mathfrak{p}}}
\newcommand{\Mp}{M_{\mathfrak{p}}}

\newcommand{\ovR}{\overline{R}}
\newcommand{\ovM}{\overline{M}}

\newcommand{\Det}{\operatorname{det}}

\newcommand{\Mat}{\operatorname{Mat}}

\newcommand{\br}[1]{\lbrack #1 \rbrack}

\newcommand{\bb}{\mathbf{b}}
\newcommand{\eb}{\mathbf{e}}

\newcommand{\tb}{\mathbf{t}}
\newcommand{\rb}{\mathbf{r}}
\newcommand{\cb}{\mathbf{c}}

\newcommand{\zerob}{\mathbf{0}}

\newcommand{\mb}{\mathbf{m}}
\newcommand{\nb}{\mathbf{n}}
\newcommand{\lambdab}{\boldsymbol{\lambda}}
\newcommand{\omb}{\overline{\mathbf{m}}}

\newcommand{\half}{1\sfrac{1}{2}}

\newcommand{\ann}{\operatorname{ann}}

\newcommand{\ia}{\mathfrak{a}}
\newcommand{\ib}{\mathfrak{b}}
\newcommand{\ic}{\mathfrak{c}}

\newcommand{\ip}{\mathfrak{p}}

\makeatletter
\def\paragraph{\@startsection{paragraph}{4}%
  \z@\z@{-\fontdimen2\font}%
  {\normalfont\bfseries}}
\makeatother

\makeatletter
\@namedef{subjclassname@2020}{%
  \textup{2020} Mathematics Subject Classification}
\makeatother

\title{Equivalent generating vectors of finitely generated modules over commutative rings}
\address{EPFL ENT CBS BBP/HBP. Campus Biotech. B1 Building, Chemin des mines, 9\\Geneva 1202, Switzerland}
\email{luc.guyot@epfl.ch}
\author{Luc Guyot}
\date{\today}
\begin{document}%%
%%%%%%%%%%%%%%%%%%
\maketitle
\begin{abstract}
Let $R$ be a commutative ring with identity and let $M$ be an $R$-module which is generated by $\mu$ elements but not fewer. 
We denote by $\SL_n(R)$ the group of the $n \times n$ matrices over $R$ with determinant $1$. 
We denote by $\E_n(R)$ the subgroup of $\SL_n(R)$ generated by the the matrices which differ from the identity by a single off-diagonal coefficient.  
Given $n \ge \mu$ and $G \in \left\{\SL_n(R),\E_n(R)\right\}$, we study the action of $G$ by matrix right-multiplication on $\V_n(M)$, the set of elements of $M^n$ whose components generate $M$. 
Assuming that $M$ is finitely presented and that $R$ is an elementary divisor ring or an almost local-global coherent Prüfer ring,
we obtain a description of $\V_n(M)/G$ which extends the author's earlier result on finitely generated modules over quasi-Euclidean rings \cite{Guy17a}. 
\end{abstract}

\section{Introduction} \label{SecIntro}
Rings are supposed unital and commutative. 
The unit group of a ring $R$ is denoted by $R^{\times}$. Let $M$ be a finitely generated $R$-module. 
We denote by $\mu(M)$ the minimal number of generators of $M$.
For $n \ge \mu(M)$, we denote by $\V_n(M)$ the set of \emph{generating vectors of $M$ of length $n$}, i.e., 
the set of $n$-tuples in $M^n$ whose components generate $M$. 
We consider the action of $\GL_n(R)$ on $\V_n(M)$ by matrix right-multiplication. 
Let $\SL_n(R)$ be the subgroup of $\GL_n(R)$ of the matrices with determinant $1$.
Let $\E_n(R)$ be the subgroup of $\SL_n(R)$ generated by the \emph{elementary matrices}, i.e., 
the matrices which differ from the identity by a single off-diagonal element. 
Let $G$ be a subgroup of $\GL_n(R)$. Two generating vectors $\mb$, $\mb' \in \V_n(M)$ are said to be \emph{$G$-equivalent}, 
which we also denote by $\mb' \sim_{G} \mb$, if there exists $g \in G$ such that 
$\mb' = \mb g$. 
Our chief concern is the description of the quotient $\V_n(M)/G$ with $G \in \left\{
\SL_n(R), \E_n(R), \GL_n(R)\right\}$ for $R$ in a class of rings comprising the Dedekind rings.  

Let us highlight earlier results pertaining to this topic.
When $M = R$, the elements of $\V_n(M)$ are usually called the unimodular rows of size $n$
and the orbit set $\W_n(R) \Doteq \V_n(R)/\E_n(R)$ has been extensively studied \cite[Section VIII.5]{Lam06}. 
Suslin and Vaserstein \cite[Théorème 4]{Bas75} \cite[Corollary 7.4]{SV76} discovered that $\W_3(R)$ can be identified with
 the elementary symplectic Witt group of $R$ if $\E_n(R)$ acts transitively on 
$\V_n(R)$ for every $n > 3$. 
In \cite[Theorem 3.6]{vdK83}, van der Kallen used the former structure inductively to define a structure of Abelian group on 
$\W_n(R)$ when $R$ has finite stable rank and $n$ is large enough. 
The interested reader is referred to \cite{DTZ18} for the latest developments on the van der Kallen group structure of $\W_n(R)$ for an affine domain $R$.
In the latter article, Das, Tikader and Zinna describe moreover the Abelian group structure of $\V_{n + 1}(R)/\SL_{n + 1}(R)$ induced by $\W_{n + 1}(R)$ for 
$R$ a smooth affine real algebra of dimension $n \ge 2$ \cite[Theorem 1.2]{DTZ18}; see also \cite{Fas11} for seminal results of this flavor.
 
For $R = S\br{x_1, \dots, x_m}$, the ring of polynomials in $m$ indeterminates over a ring $S$, 
Murthy determined conditions under which $\GL_n(R)$ acts transitively on $\V_n(\sum_{k = 1}^r Rx_i)$ for $r \le n \le m$  \cite{Mur03}. 

To introduce further known results, let us define the \emph{determinant} $\det(M)$ of a finitely generated $R$-module 
$M$ as the exterior product $\bigwedge^{\mu}M$ where $\mu = \mu(M)$. 
The module $\det(M)$ is a cyclic $R$-module whose annihilator is $\Fitt_{\mu - 1}(M)$, the $\mu$-th Fitting ideal of $M$ (see Section \ref{SubSecFittAndDet} for definition).
The \emph{determinant map} $\det: M^n \rightarrow \bigwedge^n M$ is defined by $\det(m_1, \dots, m_{\mu}) = m_1 \wedge \cdots \wedge m_n$. 
It is immediate to check that $\det(\V_{\mu}(M)) \subseteq \V_1(\det(M))$. 
For $I$ an ideal of a ring $R$, we showed that $\V_2(I)/\SL_2(R)$
 is equipotent with a subset of $\V_1(\det(I))$ \cite[Theorem A]{Guy20} if $I$ can be generated by two elements which are not zero-divisors. 
For $R$ a quasi-Euclidean ring and $M$ an arbitrary finitely generated $R$-module, 
we proved that the determinant map induces a bijection from 
$\V_{\mu}(M)/\E_{\mu}(R)$ onto $\V_1(\det(M))$ and that $\E_n(R)$ acts 
transitively on $\V_n(M)$ for every $n > \mu$ \cite[Theorem A and Corollary B]{Guy17a}. 
In this article we generalize  \cite[Theorem A and Corollary B]{Guy17a}
to two classes of rings: the elementary divisor rings and the rings considered by Couchot in \cite[Theorems 1 and 2]{Cou07}. 
For the time being, we denote by $\mathcal{D}$ the union of these two classes. Their definitions are postponed to Section  \ref{SecEDRAndArithmetic}.  
The reader should only bear in mind that $\mathcal{D}$ contains the Dedekind rings and some natural generalizations.

The rings in $\mathcal{D}$ share three features which make them particularly amenable 
to the study of the generating vectors of their finitely presented modules. 
For any such ring $R$, the following hold (see Proposition \ref{PropInstrumental} and Section \ref{SecEDRAndArithmetic}'s theorems):
\begin{itemize}
\item[(1)] Every finitely generated ideal of $R$ can be generated by two elements.
\item[(2)] The stable rank of $R$ is at most $2$.
\item[(3)] Finitely presented $R$-modules have an \emph{invariant decomposition}.
\end{itemize}

Our first result relates $\SL_n(R)$-orbits of generating vectors to the image of the determinant map.

\begin{Atheorem} \label{ThSL}
Let $R$ be a ring in $\mathcal{D}$.
Let $M$ be a finitely presented $R$-module and set $\mu = \mu(M)$. 
Then the determinant map induces a bijection from $\V_{\mu}(M)/\SL_{\mu}(R)$ onto $\V_1(\det(M))$. 
In addition, the group $\SL_n(R)$ acts transitively on $\V_n(M)$ for every $n > \mu$.
\end{Atheorem}

In order to generalize Theorem \ref{ThSL} to direct products with factors in $\mathcal{D}$, we introduce the following definition.

\begin{Adefinition}
\label{DefDelta}
A finitely generated $R$-module $M$ is said to satisfy the property $\Delta_{\SL}$ if the  determinant map
 induces a bijection from $\V_n(M)/\SL_n(R)$ to $\V_1(\bigwedge^n M)$ for every $n \ge \mu(M)$.
The ring $R$ is said to satisfy  the property $\Delta_{\SL}$ if every finitely presented $R$-module satisfies $\Delta_{\SL}$.
\end{Adefinition}

We define likewise the property $\Delta_{\E}$ for which an analog of Proposition \ref{PropStability} 
below holds (see Propositions \ref{PropProduct} and \ref{PropQuotient}).
%The following proposition outlines two stability properties of rings satisfying $\Delta_{\SL}$.

\begin{Aproposition}[Propositions \ref{PropProduct} and \ref{PropQuotient}] \label{PropStability}
The two following hold.
\begin{itemize}
\item[$(i)$]
Let $\{R_x\}_{x \in X}$ be a family of rings satisfying $\Delta_{\SL}$ and let $R =  \prod_x R_x$.
Then $R$  satisfies $\Delta_{\SL}$.
\item[$(ii)$] Let $R$ be a ring which satisfies $\Delta_{\SL}$. Let $\ovR$ be a quotient of $R$ such that the natural map $\SL_n(R) \rightarrow \SL_n(\overline{R})$ 
is surjective for every $n$. Then $\ovR$ satisfies $\Delta_{\SL}$.
\end{itemize}
\end{Aproposition}

Proposition \ref{PropStability} implies that a direct product of rings, 
each factor of which is in $\mathcal{D}$ or is the quotient of a ring in $\mathcal{D}$, has property $\Delta_{\SL}$.

Let us now outline a class of rings $R$ for which the action of $\SL_n(R)$ on $\V_n(M)$ is always transitive. 
We say that a ring $R$ has property $T_{\SL}$ if for every finitely presented $R$-module $M$,  the group $\SL_n(R)$ acts transitively on $\V_n(M)$ for every $n \ge \mu(M)$.  
We define likewise the property $T_G$ for $G \in \{\E, \GL \}$.
If $R$ satisfies $T_{\SL}$, then the only unit of $R$ is $1$, its identity element. 
Commutative rings with no units other than $1$ have been studied by Cohn \cite[Theorem 3]{Coh58}.
For every Boolean ring $B$, both $B$ and $B[x]$ have no units other than $1$. Since $B$ and $B[x]$ are elementary divisor rings 
\cite[Example 2]{Shores74}, both enjoy property $T_{\SL}$ by Theorem \ref{ThSL}.

Regarding the action of $\GL_n(R)$ on $\V_n(R)$, we observe that it is essentially determined 
by the action of $\SL_n(R)$. Indeed, decomposing $\GL_n(R) \simeq \SL_n(R) \rtimes R^{\times}$, we easily see that 
$\V_n(M)/\GL_n(R) = (R^{\times})^n \backslash \V_n(M) /\SL_n(R)$ where $(R^{\times})^n$ acts on the left via componentwise multiplication. 
If for instance an $R$-module $M$ has property $\Delta_{\SL}$, then $\V_n(M) / \GL_n(R)$ identifies with 
$$R^{\times} \backslash \V_1(\det(M)) \simeq (R/\Fitt_{\mu - 1}(M))^{\times} / (R^{\times} + \Fitt_{\mu - 1}(M)).$$
If a ring $R$ has the property $T_{\GL}$, then obviously every surjective ring homomorphism $R \twoheadrightarrow \overline{R}$ 
must induce a surjective homomorphism $R^{\times} \twoheadrightarrow (\overline{R})^{\times}$ 
on unit groups. The latter condition is equivalent to say that $R$ has stable rank $1$, see \cite[Lemma 6.1]{EO67} or \cite[Proposition 4.2]{Guy20}. 
Suppose $R$ has property $\Delta_{\SL}$, e.g., $R$ lies in $\mathcal{D}$. 
Then it is also clear that $R$ has property $T_{\GL}$ if $R$ has moreover stable rank $1$. 
For instance, the ring in \cite[Example 7]{Chen17b} is a Jacobson principal ideal domain of stable rank $1$ and thus enjoys property $T_{\GL}$. 

Besides the action of a subgroup $G \subseteq \GL_n(R)$ on $\V_n(M)$, it is also natural to consider the action of 
$\Aut_R(M)$, the group of the $R$-automorphisms of $M$, acting diagonally on $M^n$ from the left. 
This action commutes with the action of  $G$ from the right, so that $\Aut_R(M) \times G$ acts also on $\V_n(M)$ from the left. 
The following is an immediate consequence of Theorem \ref{ThSL} and of the invariant decompositions of the modules at play 
(see Theorems \ref{ThEDRDecomposition} and \ref{ThInvariantDecomposition} below):

 \begin{Acorollary}
 Let $R$ be a ring in $\mathcal{D}$. Let $M$ be a finitely presented $R$-module.
 Then $\Aut_R(M) \times \SL_n(R)$ acts transitively on $\V_n(M)$ for every $n \ge \mu(M)$.
 \end{Acorollary} 
 
 We finally consider the action of $\E_n(R)$ on generating vectors.
 Evidently, Theorem \ref{ThSL} provides us with a description of $\V_n(M)/\E_n(R)$ for $n \ge \mu(M)$ if $\SL_n(R) = \E_n(R)$. 
 If the latter identity holds, the ring $R$ is termed a \emph{$\GE_n$-ring}. If $R$ is $\GE_n$ ring for every $n \ge 2$, then $R$ is termed a \emph{$\GE$-ring}. 
 The class of $\GE$-rings includes  quasi-Euclidean rings, rings of stable rank $1$ and 
 Dedekind domains of arithmetical type with infinitely many units \cite{Vas72}. 
 By taking direct products of $\GE$-rings, we can obtain non-$\GE$-rings for 
 which $\E_n(R)$ can replace $\SL_n(R)$ in Theorem \ref{ThSL}.
 It is well-known that the class of $\GE_n$-ring is stable under the formation of direct products with finitely many factors \cite[Theorem 3.1]{Coh66}, but not for infinitely many ones. 
 A typical counter-example is  $\prod_{x} \mathbb{Z}$ where $x$ ranges in an infinite set (use, e.g., \cite[Lemma 5.1]{Coh66}). Although the latter ring is not $\GE_2$, it is a $\GE_n$-ring for every $n > 2$.  
This follows from \cite[Main Theorem]{CK83} which implies more generally that an arbitrary direct product of rings of integers or their localizations, is a $\GE_n$-ring for every $n > 2$. 
 For such products, combining Theorem \ref{ThSL} with Proposition \ref{PropStability} yields a description of $\V_{\mu}(M)/\E_{\mu}(R)$ provided that $\mu = \mu(M) > 2$.

If $R$ is not a $\GE_n$-ring, describing $\V_n(M)/\E_n(R)$ is already challenging when we restrict to $M = R$ and $n = 2$. In this case, 
$\V_2(R)/\E_2(R)$  identifies with $\E_{12}(R)\backslash \SL_2(R) / \E_2(R)$ where 
$\E_{12}(R)$ is the group of matrices of the form 
$
\begin{pmatrix}
1 & r \\
0 & 1
\end{pmatrix}
$
with $r \in R$ (see Lemma \ref{LemRPlusCyclicE}).
If for instance $R$ is the ring of integers of an imaginary quadratic field $\Q(\sqrt{-d})$ that is not $\GE_2$,
 then it follows from \cite[Theorem 2.4]{FF88} together with 
 the Normal Form Theorem for Free Products with Amalgamation \cite[Theorem IV.2.6]{LS77} that $\V_2(R)/\E_2(R)$ is infinite; 
 see \cite[Theorem 1.5]{Nic11} and \cite{She17} for an elementary proof of this fact. 
The quotient $\V_2(R)/\E_2(R)$ clearly surjects on the quotient of $\SL_2(R)$ by the subgroup of unipotent matrices. For infinitely 
many $d > 0$, this quotient surjects in turn on a non-Abelian free group \cite[Proof of Theorem 1.2]{GMV94}.
With our last result we depart from such intricate situations, assuming that $\mu(M) > 2$ and that $M$ has a non-trivial torsion submodule.
Under these assumptions, we can describe $\V_n(M)/\E_n(R)$ by means of the determinant map and the special Whitehead group $\SK_1(R)$ of $R$ 
(see Section \ref{SubSecDirectProductsOfRings} for the definition of $\SK_1$).
Let us denote by $\tau(M)$ the \emph{torsion submodule} of $M$, that is, 
$$\tau(M) \Doteq \{ m \in M \, \vert \, rm = 0 \text{ for some } r \in R \text{ which is not a zero divisor}\}.$$
A ring $R$ is a \emph{Bézout ring} if its finitely generated ideals are principal. 
\begin{Atheorem} \label{ThSK1}
Assume that $R$ is a Bézout ring whose proper quotients have stable rank $1$ or 
that  $R$ is an almost local-global coherent Prüfer ring (see Section \ref{SecEDRAndArithmetic} for definitions). 
Let $M$ be a finitely presented $R$-module. 
Let $e$ be an idempotent of $R$ such that $\mu(eM) = \mu(M)$ and set $\mu' = \mu(\tau(M))$, $\mu = \mu(M)$. 
Then the following hold.
\begin{itemize}
\item[$(i)$]
If $\mu' > 1$ or $\mu = 1$, then the determinant map induces a bijection from  from $\V_{\mu}(M)/\E_{\mu}(R)$ onto $\V_1(\det(M))$.
\item[$(ii)$]
If $\mu' = 1$ and $\mu >  2$, then $\V_{\mu}(M)/\E_{\mu}(R)$ is equipotent with $\V_1(\det(M)) \times \SK_1(R)$.
\item[$(iii)$]
If either $\mu' \neq 0$, or $\mu' = 0$ and $eM$ surjects onto a non-principal invertible ideal of $Re$, then $\E_n(R)$ acts transitively on $\V_n(M)$ for every $n > \mu$. 
This holds in particular if $R$ is an almost local-global Prüfer domain and $M$ is torsion-free but not free.
\item[$(iv)$] 
The group $\E_n(R)$ acts transitively on $\V_n(M)$ for every $n > \mu + 1$.
\end{itemize}
\end{Atheorem}
A Bézout ring $R$ such that the proper quotients of $R/\Jac(R)$ have stable rank $1$ is an elementary divisor ring \cite[Proposition 3.5 and Theorem 3.7]{McGov08}. 
To the best of our knowledge, all known examples of elementary divisor rings satisfy this property (this is corroborated by \cite[Remark 4.7]{McGov08}).
Combining Theorem \ref{ThSK1} with Proposition \ref{PropERank} below, we gain thus insight on all such examples.

\paragraph{Final Remarks} 
The proof Theorem \ref{ThSL} hinges on the fact that ideals are two-generated for rings in $\mathcal{D}$ (see specifically the results of Section \ref{SecI}).
The class of Bass rings is another well-studied class of rings whose ideals are two-generated \cite[§4]{Lam00} 
and whose finitely generated torsion-free modules are well-understood \cite{LW85}.
The Bass rings constitute therefore a class of choice to pursue
 the study of equivalent generating vectors in the spirit of Theorems \ref{ThSL} and \ref{ThSK1}.

\paragraph{Layout} This paper is organized as follows. Section \ref{SecNotation} introduces notation and background results. 
It encloses some of the fundamental properties of the determinant map, of the Bass stable rank 
and of the class $\mathcal{D}$. Section \ref{SecI} is dedicated to the study of the action of $\SL_n(R)$ on $\V_n(I)$ for $I$ a  two-generated ideal of $R$. 
Section \ref{SecDirectProducts} addresses the finitely generated modules which admit an invariant decomposition 
in the sense of Theorems \ref{ThEDRDecomposition} and \ref{ThInvariantDecomposition}. 
 However, the 
idempotents appearing in the decomposition of Theorem  \ref{ThInvariantDecomposition} are only handled in 
Section \ref{SubSecDirectProductsOfRings} where Proposition \ref{PropStability} is proven. 
Section \ref{SecProofs} is dedicated to the proofs of the Theorems \ref{ThSL} and \ref{ThSK1}.
 
\paragraph{Acknowledgments}
We are grateful to François Couchot and Jean Fasel for their encouragements. 
We are thankful to Justin Chen, Henri Lombardi, Bogdan Nica and Wilberd van der Kallen 
for helpful comments and references.

\section{Notation and background} \label{SecNotation}

We assume throughout that $R$ is a commutative unital ring and $M$ is a finitely generated $R$-module. 
We denote by $\mathcal{J}(R)$ the \emph{Jacobson radical} of $R$, that is, the intersection of the maximal ideals of $R$.
If $\ip$ is a prime ideal of $R$, we denote by $\Rp$ the localization of $R$ at $\ip$. 
Similarly, we denote by $\Mp$ its localization at $\ip$, which we identify with $M \otimes_R \Rp$. 
Given an element $m$ in $M$, we abuse notation in denoting also by $m$ its image $m\otimes_R 1 \in \Mp$. 
If $M$ and $N$ are two submodules of a given $R$-module, we denote by $(M:N)$ the ideal consisting of all elements $r \in R$ such that $rN \subseteq M$. 
We define the \emph{annihilator $\ann_R(M)$ of $M$} as $(0: M)$.
If $Rm$ and $Rm'$ are two cyclic submodules of $M$, we simply write $(m : m')$ for $(Rm : Rm')$. 

\paragraph{Regular elements}
An element $r$ of a ring $R$ is a \emph{zero divisor} if $\ann_R(r) \Doteq (0: r) \neq \{0\}$. An element of $R$ is \emph{regular} if it is not a zero divisor.
 An ideal of $R$ is \emph{regular} if it contains a regular element.
 The \emph{total ring of quotients $K(R)$} of $R$ is the localization $R[S^{-1}]$ where $S$ is the set of regular elements of $R$. 
 The above definition of a regular element should not be confused with the definition of a \emph{von Neumann regular element} of $ R$, that is an element for which there exists
 $a \in R$ satisfying $r = ara$. A ring whose elements are von Neumann regular is called a \emph{von Neumann regular ring}.

\paragraph{Matrices over $R$ and $M$}
 Given two $1 \times n$ row vectors $\rb = (r_1, \dots, r_n) \in R^n$ and $\mb = (m_1,\dots, m_n) \in M^n$,  
 we denote by $\rb^{\top}$ and $\mb^{\top}$ the $n \times 1$ column vectors obtained by transposition and 
 we define their product $\rb \mb^{\top} = \mb \rb^{\top} \Doteq \sum_i r_i m_i$. 
 Based on these identities, the product of any two matrices over $R$ and $M$, with compatible numbers of rows and columns, is uniquely defined. 
We denote respectively by $0_{m \times n}$ and by ${1_n}$ the $m \times n$ zero matrix and the $n \times n$ identity matrix over $R$.

\subsection{Fitting ideals and the determinant module} \label{SubSecFittAndDet}
The invariant on which our study of the action of $\SL_{\mu}(R)$ on $\V_{\mu}(M)$ hinges is the determinant map. 
This map can be effectively computed thanks to the $\mu$th-Fitting ideal.

\paragraph{Fitting ideals}
Let 
$$
F \xrightarrow[]{\varphi} G \rightarrow M \rightarrow 0 
$$
be an exact sequence where $F$ and $G$ are free $R$-modules and $G$ is finitely generated. Let $I_i(\varphi)$ 
the ideal of $R$ generated by the $i \times i$ minors of the matrix of $\varphi: F \rightarrow G$, agreeing that $I_0(\varphi) = R$. 
Then the \emph{$(i + 1)$-th Fitting ideal of $M$} is $\Fitt_i(M) \Doteq I_{\mu(G) - i}(\varphi)$. 
The Fitting ideals are independent of $\varphi$ by Fitting's lemma \cite[Corollary 20.4]{Eis95}. 
They  commute with \emph{base change}, i.e, $\Fitt_i(M \otimes_R S) = \Fitt_i(M) \otimes_R S$ for any $R$-algebra $S$ \cite[Corollary 20.5]{Eis95}, and hence with localization.

\paragraph{The determinant of a finitely generated module} 
Let $\mu = \mu(M)$.
Recall that the determinant $\det(M)$ of $M$ is the exterior product $\bigwedge^{\mu}M$. 
This is a cyclic $R$-module whose annihilator is $\Fitt_{\mu - 1}(M)$ \cite[Exercise 20.9.i]{Eis95}.
Given $\mb = (m_1, \dots, m_{\mu}) \in \V_{\mu}(M)$, we denote by $\phi_{\mb}: \det(M) \rightarrow R/\Fitt_{\mu - 1}(M)$ 
the isomorphism induced by the map $m_1 \wedge  \cdots  \wedge m_{\mu} \rightarrow 1 + \Fitt_{\mu - 1}(M)$. 
Given another generating pair $\mb' = (m_1', \dots, m_{\mu}') \in M^{\mu}$, we define 
$$\text{det}_{\mb}(\mb') \Doteq \phi_{\mb}(m_1' \wedge \dots \wedge m_{\mu}').$$
It is easy to check that $\det_{\mb}(\V_{\mu}(M))$ is a subgroup of $(R/\Fitt_{\mu - 1}(M))^{\times}$ which doesn't depend on $\mb$.
The following lemma is straightforward.

\begin{lemma} \label{LemDet}
Let  $\mb, \mb', \mb'' \in \V_{\mu}(M)$.
Then the following assertions hold.
\begin{itemize}
\item[$(i)$] $\det_{\mb}(\mb' A) = \det_{\mb}(\mb')\det(A)$, for every $\mu \times \mu$ matrix $A$ over $R$.
\item[$(ii)$] $\det_{\mb}(\mb'') = \det_{\mb}(\mb') \det_{\mb'}(\mb'')$.
\item[$(iii)$] $\det_{\mb}(\mb') = 1 + \Fitt_{\mu - 1}(M)$, if and only if, $\det_{\mb}(\mb') = 1 + \Fitt_{\mu - 1}(M_{\ip})$ 
for every maximal ideal $\ip$ of $R$, where $\mb$ and $\mb'$ denote (abusively) their natural images in $\V_{\mu - 1}(M_{\ip})$ in the left-hand side of the identity.
\end{itemize}
\end{lemma}

$\square$

When there is no risk of ambiguity, we simply denote by $1$ the identity element of $R/\Fitt_{\mu -1}(M)$.
The ideal $\Fitt_{\mu - 1}(M)$ has a convenient description which makes the computation of $\det_{\mb}$ 
effective when a workable presentation of $M$ is given. Let $\mb \in \V_{\mu}(M)$. We say that an element $r \in R$ 
is \emph{involved in a relation of $\mb$} if there is $(r_1,\dots, r_{\mu}) \in R^{\mu}$ such that $ \sum_{i = 1}^{\mu} r_i m_i = 0$ and $r = r_i$ for some $i$. 

\begin{lemma} \label{LemFittMuMinusOne}
Let $\mb \in\V_{\mu}(M)$. Then $\Fitt_{\mu - 1}(M)$  is the set of the elements of $R$ involved in a relation of $\mb$. 
\end{lemma}

$\square$
 
Let $\overline{\mb} = \pi(\mb)$ be the image of $\mb \in V_{\mu}(M)$ by the natural map $\pi : M \rightarrow M/\Fitt_{\mu - 1}(M)M$ 
 and let $\eb$ be the canonical basis of $(R/\Fitt_{\mu - 1}(M))^{\mu}$. Then the map $\overline{\mb} \mapsto \eb$ induces an 
 isomorphism $\varphi_{\mb}$ from $M/\Fitt_{\mu - 1}(M)M$ onto $(R/\Fitt_{\mu - 1}(M))^{\mu}$.
The next lemma shows how the map $\pi$ can be used to compute $\det_{\mb}$.
\begin{lemma} \label{LemDetPi}
Let $\mb, \mb' \in \V_{\mu}(M)$. Then $\det_{\mb}(\mb')$ is the determinant of the unique $\mu \times \mu$ matrix $\overline{A}$ over $R/\Fitt_{\mu - 1}(M)$ 
satisfying $\pi(\mb') = \pi(\mb) \overline{A}$, that is the matrix of $\varphi_{\mb} \circ \pi(\mb')$ with respect to $\eb$.
\end{lemma}

$\square$

\subsection{Ranks} \label{SubSecRanks}
Following \cite[Section 9]{Rief83} and \cite[Section 6.7.2]{McCR87}, we define the Bass stable rank of a finitely generated $R$-module $M$.
An integer $n > 0$ lies in the \emph{stable range of} 
$M$ if for every $\mb = (m_1,\dots,m_{n + 1}) \in \V_{n + 1}(M)$, there is $(r_1, \dots, r_n) \in R^n$ such that $(m_1 + r_1 m_{n + 1}, \dots, m_{n} + r _nm_{n + 1})$ belongs to $\V_n(M)$. 
If $n$ lies in the stable range of $M$, then so does $k$ for every $k > n$ \cite[Lemma 11.3.3]{McCR87}.
The \emph{stable rank $\sr(M)$ of $M$} is the least integer in the stable range of $M$. 

The Bass stable rank can be characterized in terms of a lifting property for the generating vectors of quotient modules. 

\begin{proposition} \label{PropStableRankAndSurjectivity}
Let $M$ be a finitely generated $R$-module and let $n \ge \mu(M)$. Then the following are equivalent:
\begin{itemize}
\item[$(i)$] $\sr(M) \le n$.
\item[$(ii)$] For every $R$-submodule $N \subseteq M$, the map $\V_n(M) \rightarrow \V_n(M/N)$ sending $(m_1, \dots, m_n)$ to $(m_1 + N, \dots, m_n + N)$ is surjective. 
In other words, every generating $n$-vector of $M/N$ lifts to a generating $n$-vector of $M$.
\end{itemize}
\end{proposition}
\begin{proof}
Straightforward, see the proof of \cite[Proposition 4.2]{Guy20}.
\end{proof}

Extending naturally the ranks introduced in \cite[Section 11.3.3]{McCR87} to finitely generated modules, 
we define the \emph{linear rank $\glr_R(M)$} and the  \emph{elementary rank $\er_R(M)$} of $M$ as the least integer $n \ge \mu(M)$ such that $\GL_k(R)$, 
respectively $\E_k(R)$, acts transitively on $\V_k(M)$ for every $k > n$. 
Note that the analogous definition based on $\SL_k(R)$ yields the rank $\glr_R(M)$. 
When there is no risk of ambiguity, we simply write $\glr(M)$ and $\er(M)$ instead of $\glr_R(M)$ and $\er_R(M)$.

It is easy to check that $\SL_2(R)$ acts transitively on $\V_2(R)$. Hence $\glr(R) = 2$ is equivalent to $\glr(R) = 1$. 
A ring $R$ is said to be a \emph{Hermite ring} if $\glr(R) = 1$. 
This is \cite[Definition I.4.6]{Lam06}; Lam's Hermite rings are sometimes called \emph{completable} rings.

Following \cite{Coh66}, we call $R$ a \emph{$\GE_n$-ring} if $\SL_n(R) = \E_n(R)$ and $R$ is said to be 
a \emph{generalized Euclidean ring}, or simply a \emph{$\GE$-ring}, if it is a $\GE_n$-ring for all $n > 1$.

\begin{remark}
The ring $R$ is a $\GE_2$-ring if and only if $\E_2(R)$ acts transitively on $\V_2(R)$. If $R$ is Hermite, then $R$ is a $\GE$-ring if and only if $\er(R) = 1$.  
\end{remark}

We denote by $\Jac(M)$, the \emph{Jacobson radical of $M$}, that is, the intersection of the maximal submodules of $M$ (see \cite[§24]{Lam01} for properties and examples).
\begin{proposition}{\cite[Proposition 11.3.11 and Lemma 11.4.6]{McCR87}}  \label{PropRanks}
Let $M$ be a finitely generated $R$-module and let $N$ be a submodule of $M$. Then the following hold.
\begin{itemize}
\item[$(i)$]  $\er(R) \le \sr(R)$. 
\item[$(ii)$]  $\mu(M) \le \glr(M) \le \er(M)$ and $\sr(M) \le \mu(M) + \sr(R) - 1$. 
\item[$(iii)$] $\sr(M/N) \le \sr(M)$ and equality holds if $N \subseteq \Jac(M)$.
\end{itemize}
\end{proposition}

\begin{proposition}  \label{PropERank}
Let $M$ be a finitely generated $R$-module. Then the following hold.
\begin{itemize}
\item[$(i)$] $\er(M) = \er_{R/I}(M)$ for every ideal $I \subseteq \ann(M)$. 
\item[$(ii)$] $\glr(M) = \glr_{R/I}(M)$ for every ideal $I \subseteq \ann(M) \cap \Jac(R)$. 
\item[$(iii)$] If $M \neq  \{0\}$, then $\er(M) \le \mu(M) + \sr(M) - 1$. 
\item[$(iv)$]  $\glr(M/N) = \glr(M)$ and $\er(M/N) = \er(M)$ for every $R$-submodule $N \subseteq \Jac(M)$.
\item[$(v)$] If $M \neq \{0\}$ and $\sr(R) = 1$ then $\er(M) = \sr(M) = \mu(M)$.
\end{itemize}
\end{proposition}

\begin{proof}
$(i)$. Observe that the natural map $\E_n(R) \rightarrow \E_n(R/I)$ is surjective for every ideal $I$ of $R$ and every $n > 0$. The result follows immediately.
$(ii)$. Observe that the natural map $\GL_n(R) \rightarrow \GL_n(R/I)$ is surjective for every ideal $I \subseteq \Jac(R)$ and every $n > 0$ \cite[Exercise I.1.12.iv]{Wei13}.

$(iii)$. We can assume, without loss of generality, that $\sr(M) < \infty$.  Let $n > \mu(M) + \sr(M) - 1$ and let $\mb \in \V_n(M)$, 
$\mb' = (m_1', \dots, m_{\mu}', 0, \dots, 0) \in \V_n(M)$ with $\mu \Doteq \mu(M)$. Let $s \Doteq \sr(M)$. It readily follows from the definition of $s$ that 
$\mb \sim_{\E_n(R)} (m_1'', \dots, m_s'', 0, \dots, 0)$ for some $(m_1'', \dots, m_s'') \in \V_s(M)$. 
The following equivalences are then immediate.

\begin{eqnarray*}
\mb  & \sim_{\E_n(R)} & (m_1'', \dots, m_s'', m_1', \dots, m_{\mu}', 0, \dots, 0),\\
        & \sim_{\E_n(R)}  & (0, \dots, 0, m_1', \dots, m_{\mu}', 0, \dots, 0),\\
        & \sim_{\E_n(R)}  & \mb'.
\end{eqnarray*}
This shows that $\er(M) < n$ and hence $\er(M) \le \mu + s - 1$.

$(iv)$. The inequality $\glr(M) \ge \glr(M/N)$ is trivial. To prove the reverse inequality, let $\mu \Doteq \mu(M) = \mu(M/N)$ and $n > \glr(M/N)$. 
Let $\mb \in \V_n(M)$ and $\mb' = (m_1', \dots, m_{\mu}', 0, \dots, 0) \in \V_n(M)$. 
By hypothesis, we can find $\nb = (\nu_1, \dots, \nu_n) \in N^n$ such that $\mb \sim_{\GL_n(R)} \mb' + \nb$. 
As $(m_1' + \nu_1, \dots, m_{\mu}' + \nu_{\mu})$ generates $M$, we easily derive that $\mb' + \nb \sim_{\GL_n(R)} \mb'$.
Thus we established that $\mb \sim_{\GL_n(R)} \mb'$ for every $\mb \in \V_n(R)$, which proves that $\glr(M) \le \glr(M/N)$. 
The proof of the analog result for the elementary rank is identical.

$(v)$. The equality $\sr(M) = \mu(M)$ is given by Proposition \ref{PropRanks}.$ii$.
We shall prove that $\er(M) = \mu(M)$ by induction on $\mu \Doteq \mu(M) > 0$. If $\mu = 1$, then $M \simeq R/ \ann(M)$. 
As $\er(M) = \er_{R/\ann(M)}(M)$ by $(i)$, we deduce from Proposition \ref{PropRanks} that $\er(M) \le \sr(R/\ann(M)) \le \sr(R)$. Therefore $\er(M) = 1$.
Assume now that $\mu > 1$ and let $n > \mu, \mb \in \V_n(M)$ and $\mb' = (m_1', \dots, m_{\mu}', 0, \dots, 0) \in \V_n(M)$. Since $n > \mu = \sr(M)$, 
we have $\mb \sim_{\E_n(R)} (m_1, \dots, m_{\mu}, m_1', 0, \dots, 0)$ for some $(m_1, \dots, m_{\mu}) \in \V_{\mu}(M)$. The induction hypothesis applied to $M/Rm_1'$ yields
$\mb \sim_{\E_n(R)} (m_1' + \lambda_1 m_1', \dots, m_{\mu}' + \lambda_{\mu} m_1', m_1', 0, \dots, 0)$ for some $\lambda_1, \dots, \lambda_{\mu} \in R$. 
Therefore $\mb\sim_{\E_n(R)} \mb'$.
\end{proof}

A ring $R$ is \emph{semi-local} if it has only finitely many maximal ideals. 
A ring $R$ is \emph{K-Hermite} if for every $n \ge 2$ and every $\rb \in R^n$, 
there is $\gamma \in \GL_n(R)$ such that $\rb \gamma = (0,\dots, 0, d)$ for some $d \in R$. 
It follows immediately from the definition that a K-Hermite ring is both a Bézout ring and a Hermite ring. 

\begin{proposition}  \label{PropStableRank}
 The following hold:
\begin{itemize}
\item[$(i)$] $\sr(R) = 1$ if $R$ is semi-local \cite[Corollary 10.5]{Bas64}.
\item[$(ii)$] $\sr(R) \le2$ if $R$ is a K-Hermite ring \cite[Proposition 8.$i$]{MeMo82}.
\item[$(iii)$] $\sr(R) \le2$ if every proper quotient of $R/\Jac(R)$ has stable rank $1$ \cite[Theorem 3.6]{McGov08}.
\item[$(iv)$] $\sr(R) \le \dim_{\text{Krull}}(R) + 1$ \cite[Corollary 2.3]{Heit84} or \cite[Theorem 2.4]{CLQ04}.
\end{itemize}
\end{proposition}

\subsection{Elementary divisor rings and coherent Prüfer rings} \label{SecEDRAndArithmetic}

In this section, we underscore the properties of the rings in $\mathcal{D}$ which 
are instrumental in the proofs of Theorems \ref{ThSL} and \ref{ThSK1}. We detail in particular the structure of finitely presented modules over these rings
by stating the corresponding invariant decomposition theorems. 
Complete definitions are provided and some remarkable examples and results are briefly outlined. 
Recall that the class $\mathcal{D}$ is the union of two classes: the elementary divisor rings and the almost local-global coherent Prüfer rings 
; the latter are the almost local-global semihereditary rings of Couchot \cite{Cou07}.

A ring $R$ is an \emph{elementary divisor ring} if every matrix over $R$ with finitely many rows and columns  
admits a Smith Normal Form, that is, for every $m \times n$ matrix $A$ over $R$, we can find 
$B \in \GL_m(R), C \in \GL_n(R)$ such that $BAC$ is a diagonal matrix whose diagonal coefficients $d_1, \dots, d_k$ 
satisfy the divisibility condition $d_{i + 1} \, \vert \, d_i$ for every $1 \le i \le k - 1$. 

The class of elementary divisor rings is easily seen to be stable under the formation of direct products and quotients. 
An elementary divisor ring is clearly a K-Hermite ring (see definition right before Proposition \ref{PropStableRank}), hence a Hermite ring. 
A finitely presented module over an elementary divisor ring decomposes into a direct sum of finitely many cyclic modules whose annihilator forms a descending chain.

\begin{theorem} \cite[Theorem 9.1]{Kap49}  \label{ThEDRDecomposition}
Let $R$ be an elementary divisor ring. Then $R$  satisfies the Invariant Factor Theorem, 
that is, every finitely presented $R$-module $M$ decomposes as 
 $$R/\ia_1 \times \cdots \times R/\ia_k \,(k \ge 0)$$ 
where the ideals  $\ia_i \subseteq R$ form a descending chain $R \supsetneq  \ia_1 \supseteq \cdots \supseteq \ia_k$. 
\end{theorem}

Conversely, a ring whose finitely presented modules have a decomposition into a direct sum of cyclic modules is an elementary divisor rings \cite[Theorem 3.8]{LLS74}.

\begin{examples} \label{ExEDR}
Principal ideal rings, i.e., rings whose ideals are principal, are elementary divisor rings \cite[Theorem 12.3]{Kap49} .
The ring of entire functions and the ring of algebraic integers are two emblematic non-Noetherian elementary divisor domains \cite[Examples 1 and 2 of Section 4]{DB72}.
A Boolean ring and more generally a von Neumann ring $R$ \cite[Remark 12]{GH56} as well as  $R[X]$ \cite[Example 2]{Shores74} are elementary divisor rings.
\end{examples}

We turn now to the definition of an \emph{almost local-global coherent Prüfer ring}.

A ring is \emph{reduced} if it has no non-zero nilpotent element.
A ring is \emph{arithmetic} if its finitely generated ideals are locally principal. 
A \emph{Prüfer} ring in the sense of Hermida and S\`{a}nchez-Giralda \cite[Definition 4]{HS86} is a reduced arithmetical ring \cite[Proposition and definition VIII.4.4]{LQ15}. 
A ring is \emph{coherent} if its finitely generated ideals are finitely presented. 
A Prüfer coherent ring is characterized by the property that its finitely generated ideals are projective \cite[Theorem XII.4.1]{LQ15}. 
A Prüfer coherent ring is often called a \emph{semihereditary} ring, a terminology we shall discontinue (following Lombardi and Quitté \cite{LQ15}) as we find it not suggestive enough.

A ring $R$ is \emph{local-global}, or \emph{LG} for short, if every (possibly multivariate) polynomial over $R$ whose values generate $R$, 
represents a unit. 
The LG rings \cite{McDW81, EG82} are a natural generalization of the rings satisfying 
van der Kallen's primitive criterion \cite[Definition 1.10]{vdK77}, see \cite[Proposition on page 456]{McDW81}.
Let $f(X) = aX + b$ with $(a, b) \in \V_2(R)$. 
Clearly, the ideal generated by the values $f(R)$ of $f$ is $R$. If $R$ is LG, then there is $x \in R$ such that $ax + b$ is a unit. Therefore LG rings have stable rank $1$. 
A ring $R$ is \emph{almost local-global} in the sense of Couchot \cite{Cou07}, or \emph{almost-LG} for short, if $R/Rr$ is LG for every regular element $r \in R$. 

\begin{examples} \label{ExLocalGlobal}
Semi-local rings and rings of Krull dimension $0$ (e.g. von Neumann regular rings) are LG rings \cite[Examples 4.1 and 4.2 of Chapter V]{FS01}, \cite[Fact IX.6.2]{LQ15}. 
The Nagata ring $R(X)$ of a ring $R$ is an LG ring \cite[Fact IX.6.7]{LQ15}.
\end{examples}

A \emph{Dedekind domain} is a Noetherian Prüfer domain. 
Dedekind domains are one-dimensional \cite[Theorem 11.6]{Mat89}, hence almost-LG. 
The classical structure theorem of Steinitz for finitely generated modules over Dedekind rings \cite[Theorem 7.48]{Mag02} generalizes to 
almost-LG coherent Prüfer rings in the following way.

\begin{theorem} \cite[Theorem 2]{Cou07} \label{ThInvariantDecomposition}
Let $R$ be an almost-LG coherent Prüfer ring. Then $R$ satisfies the Invariant Factor Theorem, that is, every finitely presented $R$-module $M$ decomposes as
$$M \simeq \prod_{i = 1}^m R/I_i \times 
\prod_{j = 1}^n \left(J_{j, 1}e_j \times \cdots \times J_{j, k_j}e_j \right).$$

where $R \supsetneq I_1 \supseteq \cdots \supseteq I_m \neq \{0\}$ are invertible ideals of $R$, $(e_1, \dots, e_n)$ is a sequence of orthogonal idempotents, $(k_1, . . . , k_n)$ 
a strictly increasing sequence of positive integers and each ideal $J_{j, k}e_j$ is an invertible ideal of $Re_j$. 
In addition, the isomorphism class of $M$ is completely determined by the following invariants:
\begin{itemize}
\item the ideals $I_1, \dots, I_m$,
\item the idempotents $e_1, \dots, e_n$, 
\item the integers $k_1, \dots, k_n$ and
\item the isomorphism class of $(J_{1, 1} \cdots J_{1, k_1}e_1) \times \cdots \times (J_{n, 1} \cdots J_{n, k_n}e_n)$.
\end{itemize}
\end{theorem}

\begin{remark}
Theorem \ref{ThInvariantDecomposition} implies in particular that $J_{j, 1}e_j \times \cdots \times J_{j, k_j}e_j \simeq 
(J_{1, 1} \cdots J_{1, k_j} e_j) \times (Re_j)^{k_j - 1}$ for every $1 \le j \le n$.
\end{remark}

A ring $R$ is said to be \emph{of finite character} if $R/I$ is semi-local for every non-zero ideal $I$ of $R$. 
Using Examples \ref{ExLocalGlobal}, it is immediate to 
check that a Prüfer domain which is one-dimensional or of finite character is an almost-LG coherent Prüfer ring but not necessarily Noetherian.

\begin{examples} \label{ExSemihereditary}
The ring of algebraic integers is a non-Noetherian Bézout domain (see Examples \ref{ExEDR}) which is one-dimensional 
by the going-up and going-down theorems \cite[Theorem 9.4]{Mat89}.
Let $p$ be a prime number, let $\Z_p$ be the ring of the $p$-adic integers and let $\Q_p$ be its fraction field.
Then the integral closure of $\Z_p$ in an algebraic closure of $\Q_p$ is a one-dimensional valuation domain,
 with value group $\Q$, hence not Noetherian \cite[Theorem 11.1]{Mat89}.
\end{examples}

We conclude this section with two properties of importance for the proofs of our theorems.

\begin{proposition} \label{PropInstrumental}
Let $R$ be a ring in $\mathcal{D}$. 
Then the two following hold.
\begin{itemize}
\item[$(i)$] Every finitely generated ideal of $R$ can be generated by two elements.
\item[$(ii)$] The stable rank of $R$ is at most $2$.
\end{itemize}
\end{proposition}

\begin{proof}
Assertion $(i)$ is obvious for a Bézout ring, hence for any elementary divisor ring.
If $R$ is an almost-LG coherent Prüfer ring, then $(i)$ follows from \cite[Theorem 2 and Lemma 10]{Cou07}. 
Assertion $(ii)$ for an elementary divisor ring is given by Proposition \ref{PropStableRank}.$ii$.
For an almost-LG coherent Prüfer ring, combine Proposition \ref{PropTransitivityPP} below and \cite[Lemma 10]{Cou07}.
\end{proof}

\section{The quotient $\V_n(I)/\SL_n(R)$ for $n \ge 2$ } \label{SecI}

In this section, $I$ denotes a two-generated ideal of $R$. 
Our results assume that $I$ enjoys at least one of the following two properties.

\begin{definition}
Let $I$ be an ideal of $R$ and let $I^{-1} \Doteq \{x \in K(R) \, | \, xI \subseteq R\}$ where $K(R)$ is the total ring of quotients of $R$.
\begin{itemize}
\item The ideal $I$ is said to be \emph{invertible} if $II^{-1} = R$.
\item The ideal $I$ is said to be \emph{$\half$-generated} if $I/Rx$ is generated by one element for every $x \in R \setminus \{ 0 \}$.
\end{itemize}
\end{definition}

It is well-known that an invertible ideal contains a regular element and that it is projective of constant rank $1$ 
(see, e.g., \cite[Theorem 11.6]{Eis95}, where this claim is proved in 11.6.c and 11.6.d without actually assuming that $R$ is Noetherian). 
A $\half$-generated ideal which contains a regular element is invertible \cite[Theorem 1]{LM88}. 
In a one-dimensional domain \cite[Theorem 3.1]{Heit76} or in a ring of finite character \cite[Theorem 3]{GH70}, invertible ideals are  $\half$-generated.

\begin{lemma} \label{LemCyclicModule}
Let $I$ and $\ia$ be ideals of $R$ with $I$ invertible. If $I/I \ia$ is cyclic 
then $I/ I \ia \simeq R/ \ia$.
\end{lemma}

\begin{proof}
Since $I/I \ia$ is cyclic, it suffices to show that $\ann(I/I \ia)  = \ia$. Clearly $\ann(I/I \ia)  = (I \ia : I)$ and the 
inclusions $(I \ia : I)  \supseteq \ia$ and $(I \ia : I) \subseteq (II^{-1} \ia: II^{-1})$ are immediate. 
To conclude, we observe that $(II^{-1} \ia: II^{-1}) = \ia$ holds because $I$ is invertible.
\end{proof}

\begin{lemma} \label{LemAlmostStableRankOne}
Let $I$ be a $\half$-generated ideal of $R$. Assume that at least one of the following holds.
\begin{itemize}
\item[$(i)$] Every proper quotient of $R$ has stable rank $1$.
\item[$(ii)$] Every proper quotient of $R/\Jac(R)$ has stable rank $1$ and $I$ is invertible.
\end{itemize}
Then $\sr(I) \le 2$.
In particular $\sr(R) \le 2$ and $R$ is Hermite.
\end{lemma}

\begin{proof}
Let $(a, b, c) \in \V_3(I)$. We need to prove the existence of $(r_1, r_2) \in R^2$ such that $(a + r_1 c, b + r_2 c) \in \V_2(I)$. 

We assume first that $(i)$ holds. The proof follows closely the lines of \cite[Theorem 3.6]{McGov08}. 
If $a = 0$, then we can take $(r_1, r_2) = (1, 0)$. Thus we can suppose that $a \neq 0$. 
If $I = Ra$, then we can take $(r_1, r_2) = (0, 0)$. Thus we can suppose that $Ra \neq I$. 
As $I$ is $\half$-generated, the non-empty set $X$ of the ideals $I'$ verifying $Ra \subseteq I' \subsetneq I$ is isomorphic to the set of proper ideals of 
$R$ containing $(a : I)$ as a set partially ordered by inclusion. Hence we can consider the intersection $J$ of the maximal  elements of $X$. 
Since $J$ contains $a$, it is a non-zero ideal. Thus $\sr(I/J) = 1$, so that we can find $r \in R$ satisfying $b + rc + J = I$. 
Reasoning by contradiction, we assume that $Ra + R(b + rc) \subsetneq I$. Then $Ra + R(b + rc)$ is contained in a 
maximal element $K \in X$. As a result, we have $J \subseteq K \subsetneq I$.  We obtain a contradiction by observing that $K \supseteq Ra + R(b + rc) + J = I$. 
Therefore we can take $(r_1, r_2) = (0, r)$.

Let us assume now that $(ii)$ holds. 
If $I\Jac(R) = \{0\}$, then $\Jac(R) = \{0\}$ so that $(i)$ is satisfied. 
Otherwise, we have $I/\Jac(R)I \simeq R/\Jac(R)$ by Lemma \ref{LemCyclicModule}. Hence $\sr(I/\Jac(R)I) \le 2$ by $(i)$.
The result follows then from Proposition \ref{PropRanks}.$iii$.
\end{proof}

\begin{proposition} \label{PropOneAndAHalf}
Let $R$ and $I$ be as in Lemma \ref{LemAlmostStableRankOne}. Then $\er(I) \le 2$.
\end{proposition}

\begin{proof}
Let $(a_1, \dots, a_n) \in \V_n(I)$ with $n > 2$. 
If any of $(i)$ or $(ii)$ holds, then $\sr(I) = 2$ by Lemma \ref{LemAlmostStableRankOne}. 
Hence we can assume, without loss of generality, that $n = 3$ and $a_1 \neq 0$. 

We assume first that $(i)$ holds. 
Let $b \in I$ be such that $I = Ra_1 + Rb$. 
Then the following equivalence holds
\begin{equation} \label{EqArrow1}
(a_1, a_2, a_3) \sim_{\El_{3}(R)} (a_1, b, 0).
\end{equation} 
Indeed, since $\sr(I/Ra_1) = \sr(R/(a_1 : I)) = 1$, we have $(a_2 + Ra_1, a_3 + Ra_1) \sim_{\El_2(R)} (b + Ra_1, Ra_1)$ 
from which we easily infer the equivalence (\ref{EqArrow1}).
Fix now $(a_1', a_2', 0) \in \V_3(I)$ with $a_1' \neq 0$ and let $b' \in I$ be such that $I = Ra_1' + Rb'$. Then we have 
\begin{equation} \label{EqArrow2}
(a_1, b, 0)  \sim_{\El_{3}(R)} (a_1', a_1, b)
\end{equation} while the two equivalences
\begin{equation} \label{EqArrow3}
(a_1', a_1, b) \sim_{\El_{3}(R)} (a_1', b', 0), \,
(a_1', a_2', 0)  \sim_{\El_{3}(R)} (a_1', b', 0)
\end{equation} can be derived from (\ref{EqArrow1}) using obvious substitutions. 
Combining (\ref{EqArrow1}), (\ref{EqArrow2}) and (\ref{EqArrow3}) yields $(a_1, a_2, a_3) \sim_{\El_{3}(R)} (a_1', a_2', 0)$, which proves the result.

Let us now assume that $(ii)$ holds. We can also assume, without loss of generality, that  $I \Jac(R) \neq \{0\}$, since otherwise $(i)$ is satisfied. 
Thus $I/I\Jac(R) \simeq R/\Jac(R)$ by Lemma \ref{LemCyclicModule}. 
As we have $\er(I) = \er(I/I\Jac(R)) = \er_{R/\Jac(R)}(I/I\Jac(R))$ by Proposition \ref{PropERank}, the result follows from $(i)$.
\end{proof}

\begin{lemma} \label{LemOneAndAHalf}
Let $I$ be an invertible ideal of $R$. Suppose there is $a \in I$ such that $a + \Jac(R)$ is contained in only finitely maximal ideals of $R$. 
Then there is $b \in I$ such that $I = Ra + Rb$.
\end{lemma}

Lemma \ref{LemOneAndAHalf} implies that an invertible ideal $I$ of $R$ is generated by $1\sfrac{1}{2}$ element if $R/\Jac(R)$ is of finite character.  
This lemma is a straightforward generalization of \cite[Theorem 3]{GH70}. We provide a proof for the convenience of the reader.

\begin{proof}
If $a = 0$, then $R$ is semi-local so the result holds by \cite[Corollary 2]{GH70}. If $I = Ra$, 
the result is obvious. Otherwise $\ib \Doteq (a : I)$ is a proper ideal of $R$ and satisfies $Ra = I\ib$. Since $Ra + \Jac(R) \subseteq \ib +  \Jac(R)$, 
there are only finitely many maximal ideals $\ip_1,\dots,\ip_n$ of $R$ containing $\ib +  \Jac(R)$. By \cite[Theorem 2]{GH70}, we can find $b \in I \setminus \bigcup_{i = 1}^n I \ip_i$. 
We have $Rb = I \ic$ for $\ic = (b: I)$ and $\ic + \Jac(R)$ is not contained in any of the ideals $\ip_i$. 
Therefore $Ra + Rb = I(\ib + \ic)$ and $R = \ib + \ic + \Jac(R)$ as the latter ideal is not contained in any maximal ideal of $R$. 
Thus $I = Ra + Rb + \Jac(R)I$ and the result follows from Nakayama's Lemma \cite[Corollary 4.8]{Eis95}.
\end{proof}

\begin{corollary} \label{CorSLTransitivity}
Assume that every proper quotient of $R/\Jac(R)$ has stable rank $1$. Let $I$ be an invertible $1\sfrac{1}{2}$-generated ideal of $R$. 
Then $\SL_n(R)$ acts transitively on $\V_n(I)$ for every $n \ge 2$.
\end{corollary}

\begin{proof}
The group $\E_n(R)$ acts transitively on $\V_n(I)$ for every $n \ge 3$ by Proposition \ref{PropOneAndAHalf}.
Since $I$ is invertible, it is a faithful projective module of constant rank $1$. 
By \cite[Corollary 3.2]{Guy20}, the group $\SL_2(R)$ acts transitively on $\V_2(I)$. 
\end{proof}

\begin{corollary} \label{CorSLTransitivityFiniteChar}
Assume that $R/\Jac(R)$ is of finite character. Let $I$ be an invertible ideal of $R$. 
Then $\SL_n(R)$ acts transitively on $\V_n(I)$ for every $n \ge 2$.
\end{corollary}

\begin{proof}
Combine Lemma \ref{LemAlmostStableRankOne}, Lemma \ref{LemOneAndAHalf} and
Corollary \ref{CorSLTransitivity}.
\end{proof}

A ring $R$ is a \emph{$PP$-ring} if every principal ideal of $R$ is projective. 
A coherent Prüfer ring is, for instance, a PP-ring. Coherent Prüfer rings can be characterized as the arithmetic PP-rings \cite[Theorem XII.4.1]{LQ15}.

\begin{proposition} \label{PropTransitivityPP}
Assume $R$ is a $PP$-ring and let $I$ be an invertible ideal of $R$. Suppose moreover that the two following hold:
\begin{itemize}
\item[(i)] $\sr(R/Rr) = 1$ for every regular element $r \in R$.
\item[(ii)] $\mu(I/Rr) = 1$ for every regular element $r \in I$.
\end{itemize}
Then $I$ is generated by two regular elements and we have:
\begin{itemize}
\item $\sr(I) \le 2$, 
\item $\er(I) \le 2$.
\end{itemize}  
In particular, we have $\sr(R) \le 2$ and $R$ is Hermite. In addition, the group $\SL_n(R)$ acts transitively on $\V_n(I)$ for every $n \ge 2$.
\end{proposition}

\begin{proof}
Since $I$ is invertible, it contains a regular element $x$. Using hypothesis $(ii)$, we infer that $I = Rx + Rx'$ for some $x' \in R$.
As $R$ is a $PP$-ring, it is additively regular \cite[Propositions 1.4 and 2.8]{Mat83}. 
Hence we can find $\lambda \in R$ such that $y \Doteq x + \lambda x'$ is regular. Thus $I$ is generated by the two regular elements $x$ and $y$.

Let us prove first that $\sr(I) \le 2$. To do so, consider $(a, b, c) \in \V_3(I)$. If $a = 0$, we can replace, without loss of generality, $a$ with $x$.
In this case, we can suppose $a$ regular so that $\sr(I/Ra) = \sr(R/(a: I)) = 1$. 
Consequently, there is $r \in R$ such that  $b + rc + Ra$ generates $I/Ra$, i.e., $(a, b + rc) \in \V_2(I)$.
If $a$ is not regular, then $(0:a) = Re$ for some idempotent $e \in R$. 
Applying the previous reasoning to $Ie$ and $I(1- e)$ in the PP-rings $Re$ and $R(1 - e)$ 
yields $(r_1, r_2) \in R^2$ such that $(a + r_1c, b + r_2c) \in \V_2(I)$.

Let us prove now that $\er(I) \le 2$.
Let $(a, b, c, a_1 \dots, a_{n - 3}) \in \V_n(I)$ with $n > 2$. Because $\sr(I) \le 2$, we can assume without loss of generality that $n = 3$ and $c = x$. 
Using the fact that $\sr(I/Rx) = \sr(R/(x: I)) = 1$, we deduce that $(a + Rx, b + Rx) \sim_{\E_2(R)} (0, y + Rx)$ 
and hence $(a, b, c) \sim_{\E_3(R)} (x, y, 0)$. Therefore $\er(I) \le 2$.

As $\SL_2(R)$ acts transitively on $\V_2(I)$ by \cite[Corollary 3.2]{Guy20}, we deduce that $\SL_n(R)$ acts transitively on $\V_n(I)$  for every $n \ge 2$.
\end{proof}
 
\section{Direct products of cyclic modules and ideals} \label{SecDirectProducts}
By Theorems \ref{ThEDRDecomposition} and \ref{ThInvariantDecomposition}, an $R$-module $M$ as in Theorems \ref{ThSL} and \ref{ThSK1} 
decomposes into a direct product of cyclic modules and ideals. 
In this section, we describe the action of $G \in \{\SL_n(R), \E_n(R)\}$ on the generating vectors of  
such a direct product under assumptions which are less restrictive than those of Theorems \ref{ThSL} and \ref{ThSK1}.

Let $\ia_1,\dots,\ia_k$ be ideals of $R$. We denote by $e_i$ ($1 \le i \le k$)  the element of the direct product 
$R/\ia_1 \times \cdots \times R/\ia_k$ whose $i$-th component is the identity of $R/\ia_i$ and whose other components are zero.
If  $\mb = (m_1,\dots, m_n) \in (M \times N)^n$ with $M = R/\ia_1 \times \cdots \times R/\ia_k$ and $N$ a finitely generated module, we let 
$$\Mat(\mb) \Doteq ((m_j)_i)$$
be the $(k + 1) \times n$ matrix whose columns are the elements $m_j = ((m_j)_i)$ with 
$(m_j)_{k  + 1} \in N$ and $(m_j)_i \in R/\ia_i$ for $1 \le i \le k$ and $1 \le j \le n$.
Clearly, we have $\Mat(\mb A) = \Mat(\mb) A$ for every $n \times n$ matrix $A$ over $R$.

\subsection{Products with two factors}

This section is dedicated to modules of the form $R/\ia \times I$ ($\ia, I \subseteq R$). 
These modules are used as bases in each induction proof of Section \ref{SecMoreThanTwo}.

\begin{lemma} \label{LemReduction}
Let $M$ be a finitely generated $R$-module and let $\mathfrak{b}$ be an ideal of $R$ contained in $\ann(M)$. 
Let $\mb = ((m_1, 0), \dots, (m_{n - 1}, 0), (m_n, 1)) \in \V_n(M \times R/\mathfrak{b})$ with $n \ge \mu(M)$.
Then $\mb \sim_{\E_n(R)} ((m_1, 0), \dots, (m_{n - 1}, 0), (0, 1))$.
\end{lemma}

\begin{proof}
It suffices to show that $m_n$ lies in the submodule of $M$ generated by $m_1, \dots, m_{n - 1}$.
Since $\mb$ generates $M \times R/\mathfrak{b}$, we can find $\rb = (r_1, \dots, r_n) \in R^n$ such that 
$\mb \rb^{\top} = (m_n, 0)$. We infer from the previous identity that $r_n \in \mathfrak{b}$ and hence $m_n = \sum_{i = 1}^{n - 1}r_i m_i$.
\end{proof}

\begin{lemma} \label{LemRPlusCyclic}
Let $M = R/\ia \times R$ with $\ia$ an ideal of $R$. Let $\mb \in \V_2(M)$.
Then $\ia = \Fitt_1(M)$ and $\mb \sim_{\SL_2(R)} (\Det_{\eb}(\mb)e_1, e_2)$
where $\eb = (e_1, e_2)$ with $e_1 = (1, 0)$ and $e_2 = (0, 1)$. 
In particular, the map $\det_{\eb}$ induces a bijection from $\V_2(M)/\SL_2(R)$ onto $(R/\ia)^{\times}$.
\end{lemma}

\begin{proof}
The fact that $\ia = \Fitt_1(M)$ is a direct consequence of Lemma \ref{LemFittMuMinusOne}.
By definition, there is a $2 \times 2$ matrix $A$ over $R$ such that $\mb = \eb A$.
Since $\SL_2(R)$ acts transitively on $\V_2(R)$, we can find $\sigma \in \SL_2(R)$ such that 
$A \sigma = \begin{pmatrix} u & v \\ 0 & 1  \end{pmatrix}$ for some $u, v \in R$.  As $\mb \sigma$ generates $M$, 
Lemma \ref{LemReduction} applies with $\mathfrak{b} = \{0\}$ so that $\mb \sim_{\SL_2(R)} (ue_1, e_2)$. 
Therefore $\det_{\eb}(\mb) = u + \Fitt_1(M) = u + \ia$, which completes the proof.
\end{proof}

\begin{proposition} \label{PropSL2AndFitt1}
Let $M$ be a two-generated faithful $R$-module such that $M_{\ip}$ is isomorphic to 
$R_{\ip}/\ia_{\ip} \times R_{\ip}$, for every maximal ideal $\ip$ of $R$, with $\ia_{\ip}$ some ideal of $R_{\ip}$. 
Let $\mb, \mb' \in \V_2(M)$.
Then the following are equivalent.
\begin{itemize}
\item[$(i)$] $\mb \sim_{\SL_2(R)} \mb'$,
\item[$(ii)$] $\det_{\mb}(\mb') = 1$.
\end{itemize}
\end{proposition}

\begin{sproof}{Proof}
The implication $(i) \Rightarrow (ii)$ follows from Lemma \ref{LemDet}.$i$.
Let us prove $(ii) \Rightarrow (i)$. By \cite[Proposition 3.1]{Guy20} and Lemma \ref{LemDet}.$iii$, we can assume that $R$ is local. 
Applying Lemma \ref{LemRPlusCyclic} and Lemma \ref{LemDet}.$ii$ yields the result.
\end{sproof}

Given an ideal $\ia$ of $R$, we denote by $\E_{12}(\ia)$ the group of matrices of the form 
$
\begin{pmatrix}
1 & r \\
0 & 1
\end{pmatrix}
$
with $r \in \ia$.

\begin{lemma} \label{LemRPlusCyclicE}
Let $M = R/\ia \times R$ with $\ia$ an ideal of $R$ and let $\eb \Doteq (e_1, e_2)$. 
For every  $\mb \in \V_2(M)$, fix $\sigma_{\mb} \in \SL_2(R)$ such that $\mb = (\Det_{\eb}(\mb)e_1, e_2) \sigma_{\mb}$. 
Then the two following hold.
\begin{itemize}

\item[$(i)$]
The map 
$\Delta_{\eb}: \mb \E_2(R) \mapsto \left(\rm{det}_{\eb}(\mb), \E_{12}(\ia) \sigma_{\mb} \E_2(R)\right)$ 
defines a bijection from $\V_2(M)/\E_2(R)$ onto 
$(R/\ia)^{\times} \times \left(\E_{12}(\ia)\backslash  \SL_2(R)/\E_2(R) \right)$.
\item[$(ii)$]
If $G$ is a normal subgroup of $\SL_2(R)$ containing $\E_2(R)$, then $\V_2(M)/G$ is equipotent with
$(R/\ia)^{\times} \times \SL_2(R)/G$.
\end{itemize}
\end{lemma}

In general $\E_2(R)$ is not a normal subgroup of $\SL_2(R)$, see  e.g.,  \cite[Theorem 1.5]{Nic11} or \cite[Claim I.8.16]{Lam06}. 
More tractable subgroups $G$ could be considered instead in Lemma \ref{LemRPlusCyclicE}.$ii$, e.g., the normal closure of 
$\E_2(R)$ in $\SL_2(R)$ or the group of $2 \times 2$ unipotent matrices over $R$.

\begin{example}
Let $R$ be the ring of integers of a totally imaginary quadratic field. If $R$ is not a $GE$-ring, then 
it follows from \cite[Theorem 2.4]{FF88} and \cite[Theorem IV.2.6]{LS77} that $\E_{12}(R)\backslash \SL_2(R)/\E_2(R)$ is infinite (alternatively use \cite[Theorem 1.5]{Nic11} and \cite{She17}). 
In particular, $\V_2(R/\ia \times R)/\SL_2(R)$ is infinite for every ideal  $\ia \subseteq R$. Note that \cite[Proof of Proposition 7.5]{Stan17} offers a 
geometric insight on $\SL_2(R)/\E_2(R)$.

\end{example}

\begin{sproof}{Proof of Lemma \ref{LemRPlusCyclicE}}
$(i)$. Given $\mb \in \V_2(M)$, the existence of $\sigma_{\mb}$ is provided by Lemma \ref{LemRPlusCyclic}. In order to see that 
$\Delta_{\eb}$ is well-defined, observe that $\det_{\eb}$ is $\E_2(R)$-invariant and that the stabilizer of 
$\mb = (\delta e_1, e_2)$ in $\SL_2(R)$ is $\E_{12}(\ia)$ for every $\delta \in (R/\ia)^{\times}$. Checking that $\Delta_{\eb}$ is bijection is straightforward.
$(ii)$. Check that $\mb G \mapsto  \left(\rm{det}_{\eb}(\mb), \sigma_{\mb} \it G \right)$ is well-defined and argue as in $(i)$.
\end{sproof}

\begin{lemma} \label{LemCyclicPlusIdeal}
Let $I$ be a two-generated ideal of $R$. Let $\ia$ be an ideal of $R$ such that $\glr_R(R/\ia) = 1$. Then $R/\ia \times I$ 
is two-generated if and only if there is $(a, b) \in \V_2(I)$ such that $(1, a)$ and $(0, b)$ generate $R/\ia \times I$. 
Moreover, for every $(a, b) \in \V_2(I)$,  the following are equivalent:
\begin{itemize}
\item[$(i)$] $R/\ia \times I$ is generated by $(1, a)$ and $(0, b)$.
\item[$(ii)$] $(b: a) + \ia = R$.
\end{itemize}
If in addition we suppose $I$ invertible, then each of the above assertions is equivalent to:
\begin{itemize}
\item[$(iii)$] $I = I\ia + Rb$.
\end{itemize}
\end{lemma}

\begin{proof}

Assuming that $\mu(R/\ia \times I) \le 2$, we pick an arbitrary $\mb \in \V_2(R/\ia \times I)$.
 By hypothesis, we can find $\sigma \in \SL_2(R)$ such that $\Mat(\mb) \sigma = \begin{pmatrix}1 & 0 \\ a & b\end{pmatrix}$ for some $(a, b) \in \V_2(I)$, 
 which proves the first part of the statement. 

$(i) \Rightarrow (ii)$. By hypothesis, there is $(\lambda, \mu) \in R^2$ such that $\lambda (1, a) + \mu (0, b) = (0, a)$, which implies  $\lambda - 1 \in (b : a)$ and $\lambda \in \ia$.
$(ii) \Rightarrow (i)$.
By hypothesis, there is $(a, b) \in \V_2(I)$, $\lambda, \nu \in R$ and $\mu \in \ia$ such that $\lambda + \mu = 1$
 and $\lambda a = \nu b$. As $(0, a) = \mu(1, a) + \nu(0, b)$, we deduce that $(1, a)$ and 
$(0, b)$ generate $R/\ia \times I$.

To complete the proof, it suffices to show $(ii) \Leftrightarrow (iii)$ under the additional hypothesis.
Multiplying both sides of $(ii)$ by $I$, we obtain $I = I\ia + I (b : a)$. Clearly $(b : a) = (b : I)$. As $I$ is assumed to be invertible, 
we have $(b : I) = bI^{-1}$, hence $I = I\ia + Rb$. Conversely, multiplying both sides of $(iii)$ by $I^{-1}$ yields $R = \ia + (b: a)$.
\end{proof}

\begin{lemma} \label{LemCyclicPlusIdealFitt1}
Let $I, \ia$ be two ideals of $R$ with $I$ invertible and $\mu(I) \le 2$. Let $M \Doteq R/\ia \times I  \subseteq R/\ia \times R$. Then $\Fitt_1(M) =  \ia$.
\end{lemma}

\begin{proof}
Let us consider two presentations 
\begin{align}
\notag
F_1 \xrightarrow[]{\varphi_1} R \rightarrow R/\ia \rightarrow 0 \\
  \notag
F_2 \xrightarrow[]{\varphi_2} R^2 \rightarrow I \rightarrow 0.
\end{align}
where $F_1$ and $F_2$ are free $R$-modules. This gives rise to a presentation 
$$
F \xrightarrow[]{\varphi} R^3 \rightarrow M \rightarrow 0. 
$$
with $F = F_1 \times F_2$ and $\varphi =\varphi_1 \times \varphi_2$.
By \cite[Corollary 20.4]{Eis95}, we have 
$\Fitt_1(M) = I_2(\varphi) =  I_1(\varphi_2)I_1(\varphi_1) + I_2(\varphi_2) = \Fitt_1(I)\ia +  \Fitt_0(I)$.
 By \cite[Propositions 20.5 and 20.6]{Eis95}, we have $\Fitt_0(I) = \{0\}$ and $\Fitt_1(I) = R$. 
Therefore $\Fitt_1(M) = \ia$.
\end{proof}

\begin{proposition} \label{PropCyclicPlusIdeal}
Let $\ia \neq 0$ be an ideal of $R$ such that $\glr_R(R/\ia) = 1$.
Let $I$ be an invertible ideal of $R$ and let $M \Doteq R/\ia \times I  \subseteq R/\ia \times R$. 
We suppose moreover that at least one of the following holds:
\begin{itemize}
\item[$(i)$] $I$ is $\half$-generated.
\item[$(ii)$] $\ia$ is regular and $I/Rr$ is cyclic for every regular element $r \in I$.
\end{itemize}

Then there is $(a, b) \in \V_2(I)$ with $a \in I\ia \setminus \{0\}$ such that $R/\ia \times I$ is generated by $\left(e_1 + a e_2,  be_2 \right)$ 
where $e_1 = (1, 0)$ and $e_2 = (0, 1)$.
If $\mb$ is a generating pair of this form and $\mb' \in \V_2(M)$, then
$$
\mb' \sim_{\SL_2(R)} (\Det_{\mb}(\mb') e_1 + ae_2, b e_2).
$$

In particular, the map $\det_{\mb}$ induces a bijection from $\V_2(M)/\SL_2(R)$ onto $(R/\ia)^{\times}$.
\end{proposition}

\begin{proof}
If $(i)$ holds, we can pick $a \in I\ia \setminus \{0\}$ such that $I = Ra + Rb$ for some $b \in I$. 
If $(ii)$ holds, we can find two regular elements $x \in \ia, y \in I$. Setting $a \Doteq xy$, we also obtain that 
 $I = Ra + Rb$ for some $b \in I$.
In particular, we have 
$I = I\ia + Rb$ so that $\mb \Doteq (e_1 + a e_2,  be_2)$ generates $M$ by Lemma \ref{LemCyclicPlusIdeal}.
Let $\mb' \in \V_2(M)$ and let $\delta  \Doteq \det_{\mb'}(\mb)$. As $\mb$ generates $M$, so does 
$\mb'' \Doteq (\delta e_1 + a e_2, b e_2)$. We claim that 
$\det_{\mb}(\mb'') = \delta$.
The equivalence
$
\mb' \sim_{\SL_2(R)}  \mb''
$ follows from our claim.  Indeed, we have $M_{\ip} \simeq R_{\ip}/ \ia R_{\ip} \times  R_{\ip}$ for every maximal ideal $\ip$
 of $R$ because $I$ is invertible, so that Proposition \ref{PropSL2AndFitt1} applies.
In order to establish our claim, we consider the natural map $\pi: M \twoheadrightarrow M/\Fitt_1(M)M$. Since 
$a \in \ia = \Fitt_1(M)$ by Lemma \ref{LemCyclicPlusIdealFitt1}, we have 
$\pi(\mb'') = \pi(\mb) \begin{pmatrix} \delta & 0 \\ 0 & 1 \end{pmatrix}$ and  Lemma \ref{LemDetPi} yields $\det_{\mb}(\mb'') = \delta$. This concludes the proof.
\end{proof}

\begin{proposition} \label{PropCyclicPlusIdealBeyondMu}
Let $I, \ia$ and $M$ be as in Proposition \ref{PropCyclicPlusIdeal}. Assume moreover that $\er_R(R/\ia) = 1$ and that at least one of the following holds:
\begin{itemize}
\item[$(i)$] $I$ is $\half$-generated and $\er_R(R/\ib) = 1$ for every non-zero $\ib \subseteq I$.
\item[$(ii)$] $R$ is a PP-ring, $\ia$ is regular, $I/\ib$ is cyclic and  $\er_R(R/\ib) = 1$ for every regular ideal $\ib \subseteq I$.
\end{itemize}
Then $\er(M) = 2$.
\end{proposition}

\begin{proof}
We assume first that $(i)$ holds and let $n > 2, \mb \in \V_n(M)$. We begin by proving that
 $\mb \sim_{\E_n(R)} (e_1 + a e_2, b e_2, 0, \dots, 0)$ for some $(a, b) \in \V_2(I)$ with $a \in I \ia \setminus \{0\}$.
Since $\er_R(R/\ia) = 1$, we deduce that 
\begin{equation} \label{EqABC} \Mat(\mb E)   =
\begin{pmatrix}
1 & 0  & \zerob \\
a & b  & \cb
\end{pmatrix}
\end{equation}
for some $E \in \SL_2(R)$, some $(a, b) \in I^2$ with $a \neq 0$ and some $\cb = (c_1, \dots, c_{n - 2})\in I^{n - 2}$.
Since $\mb E$ generates $M$, there is a row vector $\rb = (r, s, \tb) \in R^n$ with $\tb \in R^{n - 2}$ such that 
$\Mat(\mb E) \rb^{\top}= 
\begin{pmatrix}
0  \\ a
\end{pmatrix}
$. As a result, we can write $a = a r + b s +  \cb \tb^{\top}$ with $r  \in  \ia$.
 Multiplying $\mb E$ on the right by an elementary matrix if needed, we can therefore assume that $a \in I \ia$. 
Since $I$ is $1\sfrac{1}{2}$-generated, we have $\mu(I/Ra) = 1$ and hence $\er_R(I/Ra) = 1$,
 so that $(b + Ra, c_1 + Ra, \dots, c_{n - 2} + Ra)  \sim_{\E_{n - 1}(R)} (d + Ra, 0, \dots, 0)$ for some $d \in I$.
Thus we can assume, without loss of generality, that $c_i \in Ra$ for every $i \ge 1$ and $I = I \ia + Rb$. 
This shows in particular that $(a, b) \in \V_2(I)$ and we infer from Lemma  \ref{LemCyclicPlusIdeal} that $(e_1 + ae_2, be_2)$ generates $M$. 
It follows immediately that $\mb \sim_{\E_n(R)} (e_1 + a e_2, b e_2, 0, \dots, 0)$, as claimed.

Let $\mb' =  (e_1 + a' e_2, b' e_2, 0, \dots, 0)$ with $(a', b') \in \V_2(I)$ and $a' \in I \ia \setminus \{0\}$. 
We shall conclude our proof by showing that $\mb' \sim_{\E_n(R)} \mb$. We begin by observing that 
 $$
 \Mat(\mb')   \sim_{\E_n(R)}
\begin{pmatrix}
1 & 0  & 0 &  \zerob \\
a' & b' & a - a' & \zerob
\end{pmatrix}
  \sim_{\E_n(R)}
\begin{pmatrix}
1 & 0  & 0 & \zerob\\
a & b' & a - a' & \zerob 
\end{pmatrix}.
 $$
 Using again the fact that $\er_R(I/Ra) = 1$, we deduce that $(b', a - a') \sim_{\E_2(R)} (b + \lambda a, \mu a)$ 
 for some $(\lambda, \mu) \in R^2$. Since $(e_1 + ae_2, (b + \lambda a)e_2)$ generates $M$ by Lemma  \ref{LemCyclicPlusIdeal}, we have 
\begin{align} \notag
 \Mat(\mb')   \sim_{\E_n(R)}  
\begin{pmatrix}
1 & 0  & 0 & \zerob\\
a & b + \lambda a & \mu a & \zerob 
\end{pmatrix}
 \sim_{\E_n(R)}  
\begin{pmatrix}
1 & 0  & 0 & \zerob \\
a & b + \lambda a & a & \zerob
\end{pmatrix}
\end{align}
from which  $\mb' \sim_{\E_n(R)} \mb$ easily follows.

Let us assume now that $(ii)$ holds. Proceeding as before, we only need to show that given the identity (\ref{EqABC}) with $a \in I \ia$, we can further suppose that $a$ is regular. 
Since $R$ is a PP-ring, there is an idempotent $e \in R$ such that $Re = (0:a)$. Let $r$ be a regular element in $I \ia$.
Then there are $\lambda \in R, \boldsymbol{\mu} \in R^{n - 2}$ such that $re = \lambda be + \boldsymbol{\mu} \cb^{\top} e$. 
As $a + re$ is regular, we can therefore assume, without loss of generality, that so is $a$. 
From there on, the proof is the same.
\end{proof}

\subsection{Products of more than two factors} \label{SecMoreThanTwo}

In this section, we consider $R$-modules $M$ of one of the two following forms 

\begin{align}
\label{EqEDR}
R/\ia_1 \times \cdots \times R/\ia_k \, (k \ge 1),\\
\label{EqSemiHereditary}
R/I_1 \times \cdots \times R/I_k \times I \times R^l  \,(k \ge 0, l \ge 0)
\end{align} 
where $R \supsetneq  \ia_1 \supseteq \cdots \supseteq \ia_k$ is a descending chain of ideals, $R \supsetneq  I_1 \supseteq \cdots \supseteq I_k$ is a descending chain of invertible ideals,
$I$ is either $\{0\}$ or an invertible ideal and $l > 0$ if  $I = \{0\}$. We show that the conclusions of Theorems \ref{ThSL} and \ref{ThSK1} hold for such modules with less restrictive assumptions on $R$. 
This is achieved with Proposition \ref{PropTorsionFree} ($M$ as in (\ref{EqSemiHereditary}), torsion free but not free), 
Propositions \ref{PropInvariantDecomposition} and  \ref{PropInvariantDecompositionKPlusD} ($M$ as in (\ref{EqEDR})), 
Proposition \ref{PropWithIAndSomeTorsion} ($M$ as in (\ref{EqSemiHereditary}) and $k \ge 1$) 
and Proposition \ref{PropSK1Rl} ($M$ is free) below. 

We denote by $\SL_k(R) \cap \E_n(R) \,(k \le n)$ the subgroup of $\SL_k(R)$ consisting of the matrices $\sigma$ such that 
$
\begin{pmatrix}
\sigma & \zerob \\
\zerob & 1_{n - k} 
\end{pmatrix} \in \E_n(R)
$.

\begin{proposition} \label{PropTorsionFree}
Assume that $R$ is Hermite. Let $I$ be an ideal of $R$ such that $\mu(I) = 2$. 
Let $M = I \times R^k \subset R^{k + 1} \,(k \ge 1)$.
Then  the following hold. 
\begin{itemize}
\item[$(i)$] $\mu(M) = k + 2$ and $\Fitt_{k + 1}(M) = \Fitt_1(I)$.
\item[$(ii)$] The map $(a, b) \mapsto (ae_1, be_1, e_2, \dots, e_{k + 1})$ induces a bijection from \\ $\V_2(I)/\SL_2(R)$ onto $\V_{k + 2}(M) / G$ with $G = \SL_{k + 2}(R)$.
If $\er(R) \le 2$, we can take $G =  \E_{\mu}(R)$ and the above map also induces a bijection from 
$\V_2(I)/(\SL_2(R) \cap \E_{k + 2}(R))$ onto $\V_{k + 2}(M) / \E_{k + 2}(R)$.
\item[$(iii)$] For every $n > k + 2$, the map 
$$(a_1, \dots, a_{n - k}) \mapsto (a_1e_1, \dots, a_{n - k}e_{n - k}, e_{n - k + 1}, \dots,  e_n)$$
 induces a bijection from  $\V_{n - k}(I)/(\SL_{n - k}(R) \cap G)$ onto $\V_n(M) / G$ where $G = \SL_n(R)$. 
If $\er(R) \le 2$, we can take $G = \E_n(R)$.
\item[$(iv)$] If $R$ and $I$ are as in Proposition \ref{PropOneAndAHalf} or as in Proposition \ref{PropTransitivityPP}, then $\er(M) = k + 2$.
\end{itemize}
\end{proposition}

\begin{proof}
We begin by proving $(i)$ and $(ii)$ simultaneously under the unique assumption $\glr(R) = 1$.
Let $\mb \in \V_{\mu}(M)$. First, we shall prove by induction on $k \ge 0$ that $\mu = k + 2$ and that
$\mb \sim_{\SL_{\mu}(R)} (a e_1, be_1, e_2, \dots, e_{k + 1})$ for some $(a, b) \in \V_2(I)$ depending on $\mb$, which is obvious for $k = 0$.  
Suppose now that $k \ge 1$.  As $\glr(R) = 1$, we deduce that
\begin{equation} \label{EqMatReduction}
\Mat(\mb) \sim_{\SL_{\mu}(R)}
\begin{pmatrix}
A & \bb \\
\zerob & 1
\end{pmatrix}
\end{equation} with $A$ a $k \times (\mu - 1)$ matrix and $\bb$ a column vector of size $k$. 
It follows from Lemma \ref{LemReduction} that
$$
\begin{pmatrix}
A & \bb \\
\zerob & 1
\end{pmatrix}
 \sim_{\E_{\mu}(R)}
\begin{pmatrix}
A & \zerob \\
\zerob& 1
\end{pmatrix}.
$$

Our claim is then proven by applying the induction hypothesis to the generating vector $\mb' \in \V_{\mu - 1}(M')$ defined by $\Mat(\mb') = A$. 

The identity $\Fitt_{k + 1}(M) = \Fitt_1(I)$ follows from Lemma \ref{LemFittMuMinusOne} applied to $\mb = (ae_1, be_1, e_2, \dots, e_{k + 1})$.
We have already established that the map $$(a, b) \mapsto (a e_1, be_1, e_2, \dots, e_{k + 1})$$ induces 
a surjection from  $\V_2(I)/\SL_2(R)$ onto $\V_{k + 2}(M) / \SL_{k + 2}(R)$. 
In order to show that the induced map is also injective, 
let us consider $(a,b), (a',b') \in \V_2(I)$ and $\sigma \in \SL_{k + 2}(R)$ such that 
$$(a' e_1, b'e_1, e_2, e_3, \dots, e_{k + 1}) =  (a e_1, be_1, e_2, e_3, \dots, e_{k + 1}) \sigma.$$ 
Then $\sigma$ must be of the form 
$
\begin{pmatrix}
\sigma' & \zerob \\
\zerob & 1_{k - 1}
\end{pmatrix}
$
with $\sigma' \in \SL_2(R)$. Thus $(a', b') \sim_{\SL_2(R)} (a, b)$.

If in addition $\er(R) \le 2$, we show likewise that 
$\mb \sim_{\E_{k + 2}(R)} (ae_1, be_1, e_2, \dots, e_{k + 1})$ for some $(a, b) \in \V_2(I)$  by induction on $k$. 
Since $\er(R) \le 2 < k + 2$, we can use $\E_{k + 2}(R)$ to reduce $\Mat(\mb)$ in (\ref{EqMatReduction}). From there, the proof follows the same lines.

$(iii)$. Proceed by induction on $k$ as in the proof of assertion $(ii)$.

Assertion $(iv)$ follows from $(iii)$, since $\E_{n - k}(R)$ acts then transitively on $\V_{n - k}(I)$ by Propositions \ref{PropOneAndAHalf} and \ref{PropTransitivityPP}.
\end{proof}

\begin{lemma} \label{LemRankOfInvariantDecomposition}
Let $M = R/\ia_1 \times \cdots \times R/\ia_k \, (k \ge 1)$ where the ideals $\ia_i \subseteq R$
 form an descending chain $R \supsetneq  \ia_1 \supseteq \cdots \supseteq \ia_k$. 
Then $\mu(M) = k$ and $\Fitt_{k - 1}(M) = \ia_1$.
\end{lemma}
\begin{proof}
Clearly $\mu(M) \le k$. Let $\ip$ be a maximal ideal of $R$ containing $\ia_1$. Then $M/\ip M$ surjects onto $(R / \ip)^k$ which is vector space of dimension $k$ over $R/\ip$. 
Therefore $\mu(M) \ge k$ and hence $\mu(M) = k$.
The identity $\Fitt_{k - 1}(M) = \ia_1$ follows from Lemma \ref{LemFittMuMinusOne} applied to $\mb = (e_1, \dots, e_k)$.
\end{proof}

\begin{proposition} \label{PropInvariantDecomposition}
Let $M$ be as in Lemma \ref{LemRankOfInvariantDecomposition}.
Let $G \in \{\SL_n(R), \E_n(R)\}$ with $n \ge k$ and let $\eb = (e_1, \dots, e_k)$. Assume that one of the following holds:
\begin{itemize}
\item[$(i)$]
$G = \SL_n(R)$ and $\glr_R(R/\ia_i) \le i$ for every $i \ge 1$.
\item[$(ii)$]
$G = \E_n(R)$ and $\er_R(R/\ia_i) \le \max(1, i - 1)$ for every $i \ge 1$.
\end{itemize}
 Then for every $\mb \in \V_n(M)$, we have:
$$\mb \sim_{G} \left\{ 
\begin{array}{cc}
(\det_{\eb}(\mb)e_1, e_2, \dots, e_k), & \text{ if } n = k, \\
(e_1, e_2, \dots, e_k, 0, \dots, 0), & \text{ if } n > k.
\end{array}
\right.$$ 
In particular $\det_{\eb}$ induces a bijection from $\V_k(M)/G$ onto $(R/\ia_1)^{\times}$.
\end{proposition}

\begin{remark} \label{RemInvariantDecomposition}
The hypothesis $(i)$ of Proposition \ref{PropInvariantDecomposition} is clearly satisfied if $R$ is a K-Hermite ring. 
It is also satisfied if at least one of the following holds:
\begin{itemize}
\item[$(a)$] Every proper quotient of $R$ has stable rank $1$.
\item[$(b)$] The ring $R$ is as in Proposition \ref{PropTransitivityPP} and each ideal $\ia_i$ is either $\{0\}$ or regular.
\end{itemize}

Hypothesis $(ii)$ is satisfied if any of $(a)$ or $(b)$ holds and $\ia_2 \neq \{0\}$.
\end{remark}

\begin{sproof}{Proof of Proposition \ref{PropInvariantDecomposition}}
We shall proceed by induction on $k \ge 1$. If $k = 1$, the result holds trivially under any of the assumptions. 
We assume throughout that $k \ge 2$ and hypothesis $(ii)$ holds, in particular $G =  \E_n(R)$.
The proof can be trivially adapted to hypothesis $(i)$.

Consider $\mb \in \V_n(M)$ with $n \ge k$. 
Since  $n > \er_R(R/\ia_k)$, we have 
$$\Mat(\mb) \sim_{\E_n(R)}
\begin{pmatrix}
A & \bb \\
\zerob & 1
\end{pmatrix}
$$ with $A$ a $k \times (n - 1)$ matrix and $\bb$ a column vector of size $k$. 
By Lemma \ref{LemReduction}, we obtain
$$
\Mat(\mb) \sim_{\E_n(R)}
\begin{pmatrix}
A &\zerob \\
\zerob & 1
\end{pmatrix}
.$$
If $n = k$, we use the induction hypothesis for the generating vector $\mb' \in \V_{k - 1}(M')$ satisfying $A = \Mat(\mb')$ 
to infer that $\mb \sim_{\E_k(R)} (\delta e_1, e_2, \dots, e_k)$ for some $\delta \in (R/\ia_1)^{\times}$.
It follows from Lemmas \ref{LemFittMuMinusOne} and \ref{LemRankOfInvariantDecomposition} that $\delta = \det_{\eb}(\mb)$.
If $n > k$, we infer that $\mb \sim_{\E_n(R)} (e_1, \dots, e_{k - 1}, 0, \dots, 0, e_k)$ and 
using one further elementary matrix, we obtain $\mb \sim_{\E_n(R)} (e_1, e_2, \dots, e_k, 0, \dots, 0)$.
\end{sproof}

\begin{proposition} \label{PropInvariantDecompositionKPlusD}
Let $M$ be as in Lemma \ref{LemRankOfInvariantDecomposition}. Let $d \ge 0$ and $n > k + d$.
Assume that $\er_R(R/\ia_i) \le i + d$ for every $i \ge 1$.
Then for every $\mb \in \V_n(M)$, we have $\mb \sim_{\E_n} (e_1, e_2, \dots, e_k, 0, \dots, 0)$.
\end{proposition}

\begin{proof}
Proceed by induction on $k \ge 1$ as in the proof of Proposition \ref{PropInvariantDecomposition}.
\end{proof}

\begin{lemma} \label{LemReductionK}
Let $M$ be a finitely generated $R$-module and let $\ia_1, \dots, \ia_k$ be ideals of $R$ such that 
$\ia_1 + \cdots + \ia_k \subseteq \ann(M)$. Let $\mb \in \V_n(M \times R/\ia_1 \times \cdots \times R/\ia_k)$ 
with $n \ge \mu(M \times R/\ia_1 \times \cdots \times R/\ia_k)$ be such that 
$$
\Mat(\mb) = \begin{pmatrix}
\mb_0 & \mb_1 \\
\zerob & 1_k
\end{pmatrix}
$$
where $\mb_0 \in M^{n - k}$ and $\mb_1 \in M^k$.
Then every component of $\mb_1$ is an $R$-linear combination of the components of $\mb_0$. In particular, we have 
$$
\Mat(\mb) \sim_{\E_n(R)} \begin{pmatrix}
\mb_0 & \zerob \\
\zerob & 1_k
\end{pmatrix}
.$$
\end{lemma}

\begin{proof}
Let us write $\mb_0 =  (m_1, \dots, m_{n - k})$, $\mb_1 = (m_{n - k + 1}, \dots, m_n)$ and fix a component $m$ of $\mb_1$. 
Since $\mb$ generates $M \times R/\ia_1 \times \cdots \times R/\ia_k$, 
there is $\rb = (r_1, \dots, r_n) \in R^n$ such that $\Mat(\mb) \rb^{\top} = 
\begin{pmatrix}
m \\
\zerob \\
\end{pmatrix}
$.  Hence we have $m = \sum_{i = 1}^n r_i m_i$ and $r_{n - k + i}  \in \ia_i$ for $1 \le i \le k$. Since 
$\ia_i \subseteq \ann(M)$ for every $i$ by hypothesis, we actually have $m = \sum_{i = 1}^{n - k} r_i m_i$.
\end{proof}

\begin{proposition} \label{PropWithIAndSomeTorsion}
Let $I$ be a non-principal invertible ideal of $R$. 

Let $$M = R/\ia_1 \times \cdots \times R/\ia_k \times I \times R^l  \,(k \ge 1, l \ge 0)$$ where the ideals 
$\ia_i \subseteq R$ form an descending chain $R \supsetneq  \ia_1 \supseteq \cdots \supseteq \ia_k$ and $\ia_k \neq 0$.

Assume moreover that at least one of the following holds:
\begin{itemize}
\item[$(a)$] Every proper quotient of $R$ has stable rank $1$ and $I$ is  $\half$-generated.
\item[$(b)$] $R$ and $I$ are as in Proposition \ref{PropTransitivityPP} and each ideal $\ia_i$ is regular.
\end{itemize}

Let $\eb \Doteq (e_1, \dots, e_{k + l + 1})$, $(a, b) \in \V_2(I)$ with $a \in I \ia_k$.

Then the following hold:
\begin{itemize}
\item[$(i)$] $\mu(M) = k + l + 1$ and $\Fitt_{k + l}(M) = \ia_1$.
\item[$(ii)$]
For every $\mb \in \V_{k + l + 1}(M)$, we have
$$\mb \sim_G (\textnormal{det}_{\eb}(\mb)e_1, e_2, \dots, e_k, e_k + ae_{k + 1}, be_{k + 1}, e_{k + 2}, \dots, e_{k + l + 1})$$
for $G = \SL_{k + l + 1}(R)$. This holds also for $G = \E_{k + l + 1}(R)$ if $k > 1$. 
\item[$(iii)$] For every $\mb \in \V_n(M)$, we have
$$\mb \sim_{\E_n(R)} (e_1, e_2, \dots, e_k, e_k + a e_{k + 1}, be_{k + 1}, e_{k + 2}, \dots, e_{k + l + 1}, 0, \dots, 0)$$ 
if $n > k + l + 1$.
\end{itemize}

In particular, $\det_{\eb}$ induces a bijection from $\V_{k + l + 1}(M)/\SL_{k + l + 1}(R)$ onto $(R/\ia_1)^{\times}$.
\end{proposition}

\begin{proof}
$(i)$. Let $\ip$ be a maximal ideal of $R$ containing $\ia_1$. Then $M$ surjects onto 
$(R/\ip)^k \times I /I \ip \times (R/\ip)^l$. By hypothesis $I/ I \ip$ is cyclic so that $I / I \ip \simeq R/ \ip$ by Lemma \ref{LemCyclicModule}. 
Therefore $M$ surjects onto $(R/\ip)^{k + l + 1}$ and hence cannot be generated by less than $k + l + 1$ elements. Since $R/\ia_k \times I $ 
can be generated by two elements thanks to Proposition \ref{PropCyclicPlusIdeal}, $M$ can be generated by $k + l + 1$ elements. Thus $\mu(M) = k + l + 1$. Let $\mb = 
(e_1, e_2, \dots, e_k, e_k + a e_{k + 1}, be_{k + 1}, e_{k + 2}, \dots, e_{k + l + 1})$. Then $\mb$ generates $M$ 
and it follows from Lemma \ref{LemFittMuMinusOne} that $\Fitt_{k + l}(M) = \ia_1 + \Fitt_1(R/\ia_k \times I)$. 
Since $\Fitt_1(R/\ia_k \times I) = \ia_k$ by Lemma \ref{LemCyclicPlusIdealFitt1}, we obtain that $\Fitt_{k + l}(M) = \ia_1$.

$(ii)$. We shall prove by induction on $l \ge 0$ that 
$$\mb \sim_G
(\delta e_1, e_2, \dots, e_k, e_k + ae_{k + 1}, be_{k + 1}, e_{k + 2}, \dots, e_{k + l + 1})$$
for some $\delta = \delta(\mb) \in (R/\ia_1)^{\times}$. The identity $\delta = \det_{\eb}(\mb)$ is a straightforward consequence of assertion $(i)$ and Lemma \ref{LemFittMuMinusOne}.

If $l = 0 $ and $k = 1$, assertion $(ii)$ is given by Proposition \ref{PropCyclicPlusIdeal}.

If $l = 0$ and $k > 1$, Proposition  \ref{PropCyclicPlusIdealBeyondMu} yields 
$$
\Mat(\mb) \sim_{\E_{k + l + 1}(R)}
\begin{pmatrix}
A & \bb & \cb\\
\zerob & 1 & 0 \\
\zerob & a & b
\end{pmatrix}
$$
where $A$ is a $(k + l - 1) \times (n - 1)$ matrix, $\bb$ and $\cb$ are column vectors of size $k + l - 1$.
By Lemma \ref{LemCyclicModule}, the map $b \mapsto 1$ induces an isomorphism from $I/I \ia_k$ onto $R/\ia_k$.
This isomorphism allows us to identify $M/\ia_k M$ with $R/\ia_1 \times \cdots \times R/\ia_k \times R/\ia_k$. 
Under this identification the image $\overline{\mb}$ of $\mb$ in $(M/\ia_k M)^{k + l + 1}$ satisfies  
$
\Mat(\omb) =
\begin{pmatrix}
A & \bb & \cb\\
\zerob & 1 & 0 \\
\zerob & 0 & 1
\end{pmatrix}
$. By Lemma \ref{LemReductionK}, the vectors $\bb$ and $\cb$ are in the $R$-linear span of the columns of $A$.
Thus 
$
\Mat(\mb) \sim_{\E_{k + l + 1}(R)}
\begin{pmatrix}
A & \zerob & \zerob \\
\zerob & 1 & 0 \\
\zerob & a & b
\end{pmatrix}
$. We complete the proof of the induction basis by applying Proposition \ref{PropInvariantDecomposition}
 to the generating vector $\mb' \in \V_k(R/ \ia_1 \times \cdots \times R/\ia_k)$ defined by $\Mat(\mb') = A$.

 Assume now that $l > 0$. We have $\er(R) \le 2$ by Lemma \ref{LemAlmostStableRankOne} under assumption $(a)$ and
  by Proposition \ref{PropTransitivityPP} under assumption $(b)$. Thus
$$\Mat(\mb) \sim_{\E_{k + l + 1}(R)}
\begin{pmatrix}
A & \bb \\
\zerob & 1
\end{pmatrix}
\sim_{\E_{k + l + 1}(R)}
\begin{pmatrix}
A & \zerob \\
\zerob& 1
\end{pmatrix}
$$
where $A$ is a $(k + l) \times (n - 1)$ matrix, $\bb$ a column vector of size $k + l$ 
and the second equivalence is given by Lemma \ref{LemReduction}. Our claim follows now from the induction hypothesis 
applied to the generating vector $\mb' \in \V_{k + l}(R/ \ia_1 \times \cdots \times R/\ia_k \times I \times R^{l - 1})$ defined by $\Mat(\mb') = A$.

$(iii)$. We proceed again by induction on $l \ge 0$. 
Assume that $l = 0$. If moreover $k = 1$, then assertion $(iii)$ is given by Proposition \ref{PropCyclicPlusIdealBeyondMu}. 
If $k > 1$, reduce $\Mat(\mb)$ as in $(ii)$ and use Proposition \ref{PropInvariantDecomposition} to complete the proof of this induction basis.
If $l > 0$, reduce $\Mat(\mb)$ as in $(ii)$ and apply the induction hypothesis.
\end{proof}

\begin{proposition} \label{PropFree}
Let $R$ be a Hermite ring. Then the following hold.
\begin{eqnarray}
\V_2(R)/\E_2(R) & \simeq & \E_{12}(R) \backslash \SL_2(R)/\E_2(R), \label{EqE21} \\ 
\V_n(R)/\E_n(R) & \simeq &  \SL_{n- 1}(R)\E_n(R) \backslash \SL_n(R) \text{ for } n > 2, \label{EqSL} \\
\V_n(R^n)/\E_n(R) & \simeq & R^{\times} \times \left(\SL_n(R) / \E_n(R) \right) \text{ for } n \ge 2, \label{EqSLl} \\
\V_n(R^l)/\E_n(R) & \simeq & \SL_{n - l}(R)\E_n(R) \backslash \SL_n(R)   \text{ for } n > l \ge 2. \label{EqSLn}
\end{eqnarray}
\end{proposition}

\begin{proof}
(\ref{EqE21}) Apply Lemma \ref{LemRPlusCyclicE} with $\ia = R$.
(\ref{EqSL}) We have $\V_n(R) = e_1\SL_n(R)$. The stabilizer of 
$e_1$ for the action of $\SL_n(R)$ is $\SL_{n - 1}(R)$. As $\E_n(R)$ is normal in $\SL_n(R)$ by Suslin's Normality Theorem \cite[Theorem 10.8]{Mag02}, the result follows. 
(\ref{EqSLl}) By Proposition \ref{PropInvariantDecomposition}, we have $\V_n(R) = \bigsqcup_{u \in R^{\times}} (ue_1, e_2, \dots, e_n)\SL_n(R)$. As the stabilizer of 
$(ue_1, e_2, \dots, e_n)$ for the action of $\SL_n(R)$ is trivial, the result follows. 
(\ref{EqSLn}) By Proposition \ref{PropInvariantDecomposition}, we have $$\V_n(R) = (e_1, e_2, \dots, e_l, 0, \dots, 0)\SL_n(R).$$ The stabilizer $S$ of 
$(e_1, e_2, \dots, e_l, 0, \dots, 0)$ for the action of $\SL_n(R)$ satisfies $\SL_{n - l}(R) \subset S \subset \E_n(R) \SL_{n - l}(R)$. Since $\E_n(R)$ is normal in $\SL_n(R)$, the result follows.
\end{proof}

\begin{proposition} \label{PropSK1Rl}
Let $R$ be a Hermite ring.
\begin{itemize}
\item [$(i)$] If $n > l + \sr(R) - 1$, then $\E_n(R)$ acts transitively on $\V_n(R^l)$.
\item [$(ii)$] If $\sr(R) \le 2$, then we have
\begin{eqnarray}
\V_n(R^n)/\E_n(R) & \simeq & R^{\times} \times \SK_1(R)  \text{ for } n > 2, \label{EqRTimesSK1} \\
\V_{n + 1}(R^n)/\E_{n + 1}(R) & \simeq & \SK_1(R) \text{ for } n \ge 2. \label{EqSK1}
\end{eqnarray}
\end{itemize}
\end{proposition}

\begin{proof}
$(i)$. This is an immediate consequence of Proposition \ref{PropInvariantDecompositionKPlusD}.
$(ii)$. Combine the identifications (\ref{EqSLl}) and (\ref{EqSLn}) of Proposition \ref{PropFree} with $\SK_1$ Stability Theorem \cite[Corollary 11.19.$v$]{Mag02}.
\end{proof}

\begin{lemma} \label{LemSK1}
Let $\ia$ be an ideal of $R$ and let $M$ be a finitely generated $R$-module. Assume that at least one of the following holds.
\begin{itemize}
\item[$(i)$] $R$ is Hermite and $M = R/\ia \times R^l$ with $l > 1$.
\item[$(ii)$] $R$ and $I$ are as in Proposition \ref{PropTransitivityPP}, $\ia$ is regular and $M = R/\ia \times I \times R^l$ with $l > 0$.
\end{itemize}

Then $\ia = \Fitt_{\mu - 1}(M)$ and $\V_{\mu}(M)/\E_{\mu}(R)$ is equipotent with 
$$(R/\ia)^{\times} \times \SL_{\mu}(R) /\E_{\mu}(R)$$ where $\mu = \mu(M)$.
If in addition $\sr(R) \le 2$, e.g., if $(ii)$ holds,  then $\V_{\mu}(M)/\E_{\mu}(R)$ can be further identified with $(R/\ia)^{\times} \times \SK_1(R)$.
\end{lemma}

\begin{proof}
The identity  $\ia = \Fitt_{\mu - 1}(M)$ is given by Lemma \ref{LemRankOfInvariantDecomposition} and Proposition \ref{PropWithIAndSomeTorsion}.
If $(i)$, respectively $(ii)$ holds, then $\V_{\mu}(M) =  \bigsqcup_{u \in (R/\ia)^{\times}} \mb_u \SL_{\mu}(R)$ where $\mb_u = (u e_1, e_2, \dots, e_{\mu})$, 
respectively $\mb_u = (u e_1, e_2 + a e_3, b e_3, e_4, \dots, e_{\mu})$ and $\mu = l$, respectively $\mu = l + 2$,
by Propositions \ref{PropInvariantDecomposition} and \ref{PropWithIAndSomeTorsion}. 
In each case, it is easily checked that the stabilizer $S_u$ of $\mb_u$ for the action of $\SL_{\mu}(R)$ consists in the matrices of the form 
$$
\begin{pmatrix}
1 & \rb \\
\zerob & 1_{\mu - 1} \\
\end{pmatrix}
$$
where $\rb$ is a row vector of length $\mu - 1$ whose components lie in $\ia$. Therefore $S_u \subset \E_{\mu}(R)$.
Since  $\E_{\mu}(R)$ is normal in $\SL_{\mu}(R)$ by Suslin's Normality Theorem \cite[Theorem 10.8]{Mag02}, we deduce that 
$\V_{\mu}(M)/\E_{\mu}(R)$ is equipotent with $(R/\ia)^{\times} \times \SL_{\mu}(R)/\E_{\mu}(R)$. If $\sr(R) \le 2$, then $\SL_{\mu}(R) /\E_{\mu}(R) \simeq \SK_1(R)$
by the $\SK_1$ Stability Theorem \cite[Corollary 11.19.$v$]{Mag02}. 
Hence $\V_{\mu}(M)/\E_{\mu}(R) \simeq (R/\ia)^{\times} \times \SK_1(R)$.
\end{proof}

\section{Direct products of rings and quotient rings} \label{SubSecDirectProductsOfRings}

The proofs of Theorems \ref{ThSL} and \ref{ThSK1} rely on the existence of invariant decompositions defined in Theorems \ref{ThEDRDecomposition} and \ref{ThInvariantDecomposition}.
The results in Section \ref{SecMoreThanTwo} don't address the decomposition of a module $M$ as in Theorem  \ref{ThInvariantDecomposition} 
when non-trivial idempotents $e_i$ split the torsion-free part of $M$.
They are however sufficient to prove Theorems \ref{ThSL} and \ref{ThSK1} in full generality. Indeed, thanks to the splittings 
$R \simeq \prod_i Re_i$ and $M \simeq \prod_i e_i M$, Proposition \ref{PropProduct} below allows us to handle the submodules $e_i M$ independently.
We also prove Proposition \ref{PropStability} by establishing Propositions \ref{PropProduct} and \ref{PropQuotient} below.

Given a ring $R$, let $\E(R) \Doteq \bigcup_n \E_n(R)$ and $\SL(R) \Doteq \bigcup_n \SL_n(R)$ be the ascending unions for which 
the embeddings $\E_n(R) \rightarrow \E_{n + 1}(R)$ and  $\SL_n(R) \rightarrow \SL_{n + 1}(R)$ are defined 
through $A \mapsto \begin{pmatrix} A & 0 \\ 0 & 1 \end{pmatrix}$. The \emph{special Whitehead group} $\SK_1(R)$ of $R$ is the quotient group $\SL(R)/\E(R)$.
\begin{lemma} \label{LemNaturalMaps}
Let $\{R_x\}$ be a family of rings where $x$ ranges in a set $X$ and let $R = \prod_x R_x$. 
Then the following hold.
\begin{itemize}
\item[$(i)$] 
The natural maps  $\SL_n(R) \rightarrow  \prod_x \SL_n(R_x)$ and $\GL_n(R) \rightarrow  \prod_x \GL_n(R_x)$ are isomorphisms for every $n \ge 1$.
\item[$(ii)$]
If $X$ is finite, then the natural map 
$\E_n(R) \rightarrow  \prod_x \E_n(R_x)$ is an isomorphism for every $n \ge 1$.
\item[$(iii)$] 
Let $n \ge 2$. If every matrix in $\E_n(R_x)$ is the product of at most $\nu_n$ elementary matrices where $\nu_n < \infty$ is independent on $x$,
then the natural map $\E_n(R) \rightarrow  \prod_x \E_n(R_x)$ is an isomorphism.
\item[$(iv)$] If $X$ is finite, or if $\sr(R_x)$ is bounded and the natural map $\E_n(R) \rightarrow  \prod_x \E_n(R_x)$ is an isomorphism, then the natural map 
$$\SK_1(R) \rightarrow  \prod_x \SK_1(R_x)$$ is an isomorphism.
\end{itemize}
\end{lemma}

\begin{proof}
We first observe that the natural maps considered in the assertions $(i)$ to $(iii)$ are clearly injective.
Assertion $(i)$ is obvious.
For $(ii)$, see \cite[Proof of Theorem 3.1]{Coh66} or reason as in $(iii)$. 

$(iii)$. Let us prove that the natural map $\varphi: \E_n(R) \rightarrow  \prod_x \E_n(R_x)$ is surjective. 
Let $(E_x) \in \prod_x \E_n(R_x)$.  Write $E_x = E_{x, 1} \cdots E_{x, \nu_n}$ with $E_{x, k} = E_{i(x)j(x)}(r_{x, k}) \, (r_{x,k} \in R_x)$ 
where $E_{ij}(r)$ denotes the elementary matrix whose $(i, j)$-coefficient is $r$ and whose other off-diagonal coefficients are zero. 
Then we have $E_x = \prod_{k = 1}^{\nu_n} \prod_{1 \le i \neq j \le n} E_{ij}(r_{x, k}(i,j))$ where the factors indexed by the pairs $(i, j)$ are ordered lexicographically and 
$$r_{x, k}(i, j) = \left\{ 
\begin{array}{cc}
r_{x, k} & \text{ if }(i, j) = (i(x), j(x)), \\
0 & \text{ else.}
\end{array}\right.
$$ Set 
$E \Doteq \prod_{k = 1}^{\nu_n} \prod_{1 \le i \neq j \le n} E_{ij}(r_{kij})$, with $r_{kij} = (r_{k, x}(i, j)) \in R$. 
Then $\varphi(E) = (\gamma_x)$, which proves that 
$\varphi$ is an isomorphism.

$(iv)$. It is trivial to check that the natural map $\psi: \SK_1(R) \rightarrow  \prod_x \SK_1(R_x)$ is always surjective. 
In order to show that $\psi$ is injective, consider $\sigma \E(R) \in \ker(\psi)$ with 
$\sigma = (\sigma_x) \in \SL_n(R)$ for some $n \ge 1$. Then $\sigma_x \in \E(R_x)$ for every $x \in X$. 
If $X$ is finite, then we can find $N \ge 1$, such $\sigma_x \in \E_N(R_x)$ for every $x$.
Therefore $(\sigma_x) \in \prod_x \E_N(R_x)$ and the result follows from $(iii)$. 
If $\sr(R_x) < N$ for some $N \ge 2$, then $\sigma_x \in \E_N(R_x)$ for every $x$ by the $\SK_1$ Stability Theorem \cite[Corollary 11.19.$v$]{Mag02}. 
Thus $(\sigma_x) \in \E(R)$, provided that $\E_N(R) \rightarrow  \prod_x \E_N(R_x)$ is an isomorphism.
\end{proof}

\begin{lemma} \label{LemProduct}
Let $\{R_x\}$ be a family of rings where $x$ ranges in a set $X$. Let $M$ be module over $R \Doteq \prod_x R_x$. 
Denote by $e_x$ the element of $R$ whose component in $R_x$ is $1$ and whose other components are zero. Let 
$\varphi: M \rightarrow \prod_x e_xM$ be the $R$-homomorphism defined by $\varphi(m) = (e_x m)$. Then the following hold.
\begin{itemize}
\item[$(i)$] If $X$ is finite, then $\varphi$ is an isomorphism. 
\item[$(ii)$] If $M$ is finitely generated, then $\varphi$ is surjective.
\item[$(iii)$] If $M$ is finitely presented, then $\varphi$ is an isomorphism.
\end{itemize}
\end{lemma}

\begin{proof}
$(i)$. The identity of $R$ decomposes as $1 = \sum_x e_x$. Thus, for every $m \in M$, we have $m = 1 \cdot m = \sum_x e_x m$, which shows that $\varphi$ is injective.
If $(m_x) \in \prod_x e_xM$, we can set $m \Doteq \sum_x e_x m_x$ so that $\varphi(m) = (m_x)$. Hence $\varphi$ is surjective.

$(ii)$. Let $(m_1, \dots, m_n)$ be a generating vector of $M$. 
Let $(m_x) \in \prod_x e_xM$.
Since $(e_x m_1, \dots, e_x m_n)$ generates $e_xM$, we can find $\lambda_{i, x} \in R_x \,(1 \le i \le n)$ such that $m_x = \sum_i \lambda_{i, x} e_x m_i$.
Setting $\lambda_i \Doteq (\lambda_{i, x})_x \in R$, we observe that $(m_x) = \varphi(\sum_i \lambda_i m_i )$. Therefore $\varphi$ is surjective.

$(iii)$. By hypothesis $M$ is the cokernel of an $n \times m$ matrix $A$ associated to a generating vector $(m_1, \dots, m_n) \in M^n$. Let $\cb_1, \dots, \cb_m \in R^n$ be the columns of $A$.
Given $\lambdab = (\lambda_1, \dots, \lambda_n) \in R^n$, the element $\sum_i \lambda_i m_i$ is zero if and only if 
$\lambdab$ lies in the $R$-linear span of the columns $\cb_j$. Let $m = \sum_i \lambda_i m_i \in M$ be such that $\varphi(m) = 0$. 
Since for every $x \in X$  we have $\sum_i e_x\lambda_i m_i = 0$, 
it follows that $e_x \lambdab = \sum_j \mu_{j, x} \cb_j$ for some $\mu_{j, x} \in R$ with $1 \le j \le m$.
Setting $\mu_j = (\mu_{j, x})_x$, we observe that $\lambdab = \sum_j \mu_j \cb_j$. Therefore $m = 0$, which shows that $\varphi$ is injective. 
As $\varphi$ is also surjective by $(ii)$, it is an isomorphism.
\end{proof}

\begin{proposition} \label{PropProduct}
Let $\{R_x\}$ be a family of rings where $x$ ranges in a set $X$. Let $M$ be a finitely presented module over $R \Doteq \prod_x R_x$. 
\begin{itemize}
\item[$(i)$]
If the $R_x$-module $e_xM$ has property $\Delta_{\SL}$ for every $x \in X$, then $M$ has property $\Delta_{\SL}$.
\item[$(ii)$]
Assume that the natural map $\E_n(R) \rightarrow \prod_x \E_n(R_x)$ is an isomorphism for every $n \ge \mu(M)$.
If in addition the $R_x$-module $e_xM$ has property $\Delta_{\E}$ for every $x \in X$, then $M$ has property $\Delta_{\E}$.
\end{itemize}
\end{proposition}

\begin{proof}
We shall only consider the property $\Delta_{\SL}$. The corresponding result for $\Delta_{\E}$ can be derived in the same way.
Assume that $(i)$ holds. Let $n \ge \mu(M)$. Using the natural isomorphism $\SL_{\mu}(R) \simeq \prod_x \SL_{\mu}(R_x)$ and Lemma \ref{LemProduct}.$iii$, 
we obtain a bijection 
$\V_n(M)/\SL_n(R) \simeq \prod_x \V_n(e_xM)/ \SL_n(R_x)$. 
By hypothesis the determinant map induces a bijection 
$\V_n(e_x M)/ \SL_n(R_x) \simeq \V_1(\bigwedge^n e_x M)$ for every $x$. 
It is easily checked that $\prod_x \V_1(\bigwedge^n e_x M) \simeq \V_1(\bigwedge^n M)$ by means of the natural $R$-isomorphism 
$\prod_x \bigwedge^n e_xM \simeq \bigwedge^n M$. The result follows.
\end{proof}

\begin{proposition} \label{PropQuotient}
Let $R$ be a ring and let $\ovR$ be a quotient of $R$.
\begin{itemize}
\item[$(i)$]If $R$ satisfies $\Delta_{\E}$, then $\ovR$ satisfies $\Delta_{\E}$.
\item[$(ii)$] If $R$ satisfies $\Delta_{\SL}$ and if the natural map $\SL_n(R) \rightarrow \SL_n(\overline{R})$ is surjective for every $n$, then $\ovR$ satisfies $\Delta_{\SL}$.
\end{itemize}
\end{proposition}

\begin{remark}
The natural map $\SL_n(R) \rightarrow \SL_n(\overline{R})$ is surjective for every $n$ and every quotient $\overline{R}$ of $R$ if $\sr(R) \le 2$ \cite[Corollary 8.3]{EO67}. 
See also the partial stability result \cite[Theorem 8.3]{EO67} when $R$ is a ring of univariate polynomials over a principal ideal domain.
\end{remark}

\begin{proof}
$(i)$. The quotient map $R \twoheadrightarrow \ovR$ endows $\ovR$ with an $R$-algebra structure.
Let $\ovM$ be an $\ovR$-module and let $M$ be the $R$-module with underlying set $\ovM$ that is obtained from $\ovM$ by restricting scalars. 
We have the following commutative diagram
$$\begin{tikzcd}
\V_n(M) / \E_n(R) \arrow{r}{\det} \arrow[swap]{d}{} & \V_1(\bigwedge^n M) \arrow{d}{} \\%
\V_n(\ovM) / \E_n(\ovR) \arrow{r}{\det}& \V_1(\bigwedge^n \ovM)
\end{tikzcd}
$$
where the left vertical arrow is given by the surjectivity of the natural map $\E_n(R) \rightarrow \E_n(\ovR)$ and the right vertical arrow is given by the base change isomorphism \cite[Proposition A2.2.b]{Eis95}.
The bottom horizontal arrow is a bijection as the other arrows are bijections. $(ii)$. Proceed as in $(i)$.
\end{proof}

\section{Proofs} \label{SecProofs}

\paragraph{Proof of Theorem \ref{ThSL}}
If $R$ is an elementary divisor ring, then $R$ is K-Hermite and $M$ has an invariant decomposition in the sense of Theorem \ref{ThEDRDecomposition}.
The result follows by combining Propositions \ref{PropTorsionFree} and \ref{PropInvariantDecomposition}.

Let us suppose that $R$ is an almost-LG coherent Prüfer ring. Let $(e_i)_i$ 
be the idempotents of $R$ appearing in the invariant decomposition of $M$ in Theorem \ref{ThInvariantDecomposition}. 
Then each $Re_i$ is an almost-LG coherent Prüfer ring and $M$ is isomorphic to $\prod_i e_iM$ by Lemma \ref{LemProduct}.$i$. 
By Proposition \ref{PropProduct}.$i$, we can assume, without loss of generality, that $M$ is of the form 
$$R/I_1 \times \cdots \times R/I_k \times I \times R^l  \,(k \ge 0, l \ge 0)$$ where $R \supsetneq  I_1 \supseteq \cdots \supseteq I_k$ is a descending chain of invertible ideals
and $I$ is either $\{0\}$ or an invertible ideal. As $R$ is a coherent Prüfer ring, it is certainly a PP-ring.
If $I$ is invertible, then $R$ and $I$ satisfy the hypotheses of Proposition \ref{PropTransitivityPP} by \cite[Lemma 10]{Cou07}. In particular $\er(R) \le 2$ and $R$ is Hermite.
The result follows by combining Propositions \ref{PropTorsionFree}, \ref{PropInvariantDecomposition} and \ref{PropWithIAndSomeTorsion}.
$\square$

\paragraph{Proof of Theorem \ref{ThSK1}}
Let us assume first that $R$ is a Bézout ring whose proper quotients have stable rank $1$. 
By \cite[Theorem 3.7]{McGov08}, the ring $R$ is an elementary divisor ring of stable rank at most $2$.

$(i)$. If $\mu = 1$, the result is trivial. Let us suppose that $\mu' > 1$.
By Theorem \ref{ThEDRDecomposition}, the module $M$  is of the form 
$R/\ia_1 \times \cdots \times R/\ia_k  \,(k > 1)$ where $\ia_1 \supseteq \cdots \supseteq \ia_k$ is a descending chain of ideals such that $R \supsetneq \ia_1$ and  $\ia_2 \neq \{0\}$.
Then the result follows from Proposition \ref{PropInvariantDecomposition} and Remark \ref{RemInvariantDecomposition}.

$(ii)$. As $M$  is of the form $R/\ia_1 \times R^l  \,(l > 1)$ with $R \supsetneq \ia_1 \neq \{0\}$, the result follows from Lemma \ref{LemSK1}.

$(iii)$. As $\sr(R) \le 2$, Proposition \ref{PropInvariantDecompositionKPlusD} applies with $d = 0$ if $\mu' > 0$. 
Otherwise, there is nothing to show since $R$ is a Bézout ring. 

$(iv)$. Proposition \ref{PropInvariantDecompositionKPlusD} applies with $d = 1$.

Let us assume now that $R$ is an almost-LG coherent Prüfer ring.  

$(i)$. If $\mu = 1$, the result is trivial. Let us suppose that $\mu' > 1$.
Because of Theorem \ref{ThInvariantDecomposition}, Proposition \ref{PropProduct}.$ii$ and Lemma \ref{LemNaturalMaps}.$ii$, we can assume, without loss of generality, that  $M$  is of the form 
$R/I_1 \times \cdots \times R/I_k \times I \times R^l  \,(k > 1, l \ge 0)$ where 
$R \supsetneq  I_1 \supseteq \cdots \supseteq I_k \neq \{0\}$ is a descending chain of invertible ideals 
and $I$ is either $\{0\}$ or an invertible ideal.
Then the result follows from Proposition \ref{PropInvariantDecomposition}, 
Remark \ref{RemInvariantDecomposition} and Proposition \ref{PropWithIAndSomeTorsion}.$ii$

$(ii)$. We can assume, without loss of generality, that  $M$ is of the form 
$R/I_1 \times I \times R^l  \,(l > 0)$ where $I_1$ is an invertible ideal, $I$ is either invertible or $\{0\}$ and $l > 1$ if $I = \{0\}$. 
Then the result follows from Lemma \ref{LemSK1}.

$(iii)$.
Let us assume first that $\mu' > 0$.
We can assume without loss of generality that  $M$  is of the form 
$R/I_1 \times \cdots \times R/I_k \times I \times R^l  \,(k > 0, l \ge 0)$ where 
$R \supsetneq  I_1 \supseteq \cdots \supseteq I_k \neq \{0\}$ is a descending chain of invertible ideals and $I$ is either $\{0\}$ 
or an invertible ideal. If $\mu = 1$, then $M = R/I_1$. As $\sr(R/I_1) = 1$, Proposition \ref{PropERank}.$iii$ yields the result. If $\mu > 1$ then the result follows from 
Proposition \ref{PropCyclicPlusIdealBeyondMu}, Proposition \ref{PropInvariantDecomposition} and Proposition \ref{PropWithIAndSomeTorsion}.$iii$.

Suppose now that $\mu' = 0$ and that $eM$ surjects onto an invertible non-principal ideal of $Re$ where $e$ is an idempotent of $R$ satisfying $\mu(eM) = \mu(M)$. 
We can assume, without loss of generality, that $e$ is the last idempotent in the decomposition of Theorem \ref{ThInvariantDecomposition}.
We can also assume that  $M$ is of the form 
$I \times R^l  \,(l \ge 0)$ with $I$ invertible, in which case the result is given by Proposition \ref{PropTransitivityPP} and Proposition \ref{PropTorsionFree}.$iv$. 
Indeed, if $e'$ is an idempotent element of $R$ such that $\mu(e'M) < \mu(M)$, then $\E_n(R)$ acts transitively on
 $\V_n(e'M)$ for every $n > \mu(M)$ by Proposition \ref{PropTransitivityPP}, Proposition \ref{PropTorsionFree}.$iv$ and 
Proposition \ref{PropSK1Rl}.$i$. 
$(iv)$. By $(iii)$, we can assume that $M$ is free, so that the result follows from Proposition \ref{PropSK1Rl}.$i$. 
$\square$

\bibliographystyle{alpha}
\bibliography{Biblio}

\end{document}